\documentclass[a4paper,reqno,10pt,oneside]{amsart}

\usepackage{amsfonts,amssymb,latexsym,xspace,epsfig,graphics,color}
\usepackage{amsmath,enumerate,stmaryrd,xy}
\usepackage{bbm}
\usepackage{natbib}
\numberwithin{figure}{section}
\usepackage{subfigure}
\usepackage{dsfont}
\usepackage{a4wide}

\usepackage{tikz}
\usepackage{tkz-berge}
\usetikzlibrary{arrows,snakes,backgrounds,trees}
\usepackage{stmaryrd}

\newtheorem{theorem}{Theorem}[section]                                          
\newtheorem{prop}[theorem]{Proposition}

\newtheorem{defn}[theorem]{Definition}
\newtheorem{remark}{Remark}[section]
\newtheorem{exa}{Example}[section]

\def\T+{{\mathbb T_d^+}}

\def\A{\mathcal{A}}

\bibdata{prob}
\bibstyle{alpha} 

\title[]{On the connection between evolution algebras,\\ random walks and graphs}

\author{Paula Cadavid, Mary Luz Rodi\~no Montoya and Pablo M. Rodr\'iguez}

\date{}

\address{
\newline
\newline
Paula Cadavid and Pablo M. Rodr\'iguez
\newline
Instituto de Ci\^encias Matem\'aticas e de Computa\c{c}\~ao - Universidade de S\~ao Paulo
\newline  
Av. Trabalhador s\~ao-carlense 400 - Centro, CEP 13560-970, S\~ao Carlos, SP, Brasil
\newline
e-mail:  pablor@icmc.usp.br
\newline
e-mail: paca@ime.usp.br 
\newline
\newline
Mary Luz Rodi\~no Montoya
\newline
Instituto de Matem\'aticas - Universidad de Antioquia 
\newline
Calle 67 N$^{\circ}$ 53-108, Medell\'in, Colombia
\newline 
e-mail: mary.rodino@udea.edu.co 
}

\subjclass[2010]{05C25, 17D92, 17D99, 05C81}
\keywords{Evolution Algebra, Random walk, Graph} 
\begin{document}


%
%









\maketitle

\begin{abstract}
Evolution algebras are a new type of non-associative algebras which are inspired from biological phenomena. A special class of such algebras, called Markov evolution algebras, is strongly related to the theory of discrete time Markov chains. The winning of this relation is that many results coming from Probability Theory may be stated in the context of Abstract Algebra. In this paper we explore the connection between evolution algebras, random walks and graphs. More precisely, we study the relationships between the evolution algebra induced by a random walk on a graph and the evolution algebra determined by the same graph. Given that any Markov chain may be seen as a random walk on a graph we believe that our results may add a new landscape in the study of Markov evolution algebras.

\end{abstract}




\section{Introduction} \label{intro}

Evolution algebras are a special class of non-associative algebras introduced by Tian in \cite{tian,tian3} as an algebraic way to mimic the self-reproduction of alleles in non-Mendelian genetics. More than ten years have passed since the first papers on this topic appeared in Mathematics literature, and a lot of research effort has been devoted to explore the connections between this abstract object and concepts of other fields. We refer the reader to \cite{camacho/gomez/omirov/turdibaev/2013,casado/molina/velasco/2016,Casas/Ladra/2014,Casas/Ladra/Rozikov/2011,tian} for a survey of properties and results of general evolution algebras; to \cite{PMP,PMP2,Elduque/Labra/2015,nunez/2013,nunez/2014} for a connection between evolution algebras and graphs; and to \cite{Falcon/Falcon/Nunez/2017,Labra/Ladra/Rozikov/2014,ladra/rozikov/2013,tian3,Rozikov/Murodov/2014} for a review of results with relevance in genetics and other applications. 

In this work we are interested in studying evolution algebras related to graphs, in a sense to be specified later. The motivation to deal with this mathematical objects is due to the fact that the references cited above, and references therein, present a wide range of connections between them and many other branches of Mathematics. In particular, the strong connection to Probability Theory suggest that new methods in Applied Mathematics could be expected as a consequence of further investigation in this field.

\smallskip
An evolution algebra is defined as follows.

\begin{defn}\label{def:evolalg}
Let $\A:=(\A,\cdot\,)$ be an algebra over a field $\mathbb{K}$. We say that $\A$ is an evolution algebra if it admits a countable basis $S:=\{e_1,e_2,\ldots , e_n,\ldots\}$, such that
\begin{equation}\label{eq:ea}
\begin{array}{ll}
e_i \cdot e_i = \sum_{k} c_{ik} e_k,&\text{for any }i,\\[.2cm]
e_i \cdot e_j =0,&\text{if }i\neq j.
\end{array}
\end{equation} 
The scalars $c_{ik}\in \mathbb{K}$ are called the structure constants of $\mathcal{A}$ relative to $S$.
\end{defn}

\smallskip
A basis $S$ satisfying \eqref{eq:ea} is called natural basis of $\mathcal{A}$. $\mathcal{A}$ is real if $\mathbb{K}=\mathbb{R}$, and it is nonnegative if it is real and the structure constants $c_{ik}$ are nonnegative. In addition, if $0\leq c_{ik}\leq 1$, and 
$$\sum_{k=1}^{\infty}c_{ik}=1,$$
for any $i,k$, then $\A$ is called {\it Markov evolution algebra}. The name is due to that there is an interesting one-to-one correspondence between $\A$ and a discrete time Markov chain $(X_n)_{n\geq 0}$ with state space $\{x_1,x_2,\ldots,x_n,\ldots\}$ and transition probabilities given by $(c_{ik})_{i,k\geq 1}$, i.e., for $i,k\in\{1,2,\ldots\}$:
\begin{equation}\label{eq:tranprob}
\nonumber c_{ik}=\mathbb{P}(X_{n+1}=x_k|X_{n}=x_i),\end{equation}
for any $n\geq 0$. For the sake of completeness we remember that a discrete-time Markov chain can be thought of as a sequence of random variables $X_0, X_1, X_2, \ldots, X_n, \ldots$ defined on the same probability space, taking values on the same set $\mathcal{X}$, and such that the Markovian property is satisfied, i.e., for any set of values $\{i_0, \ldots, i_{n-1},x_n, x_{k}\} \subset \mathcal{X}$, and any $n\in \mathbb{N}$, it holds
$$\mathbb{P}(X_{n+1}=x_k|X_0 = i_0, \ldots, X_{n-1}=i_{n-1}, X_{n}=x_i)=\mathbb{P}(X_{n+1}=x_k|X_{n}=x_i).$$ 
We refer the reader to \cite{karlin/taylor,ross} for a review of Markov chains. Notice that in the correspondence between the evolution algebra $\A$ and the Markov chain $(X_n)_{n\geq 0}$ what we have is each state of $\mathcal{X}$ identified with a generator of $S$. 

\smallskip
Perhaps the main contribution of the correspondence between Markov chains and evolution algebras is that many problems coming from applied sciences, which may be currently solved by mean of probabilistic methods, i.e. stochastic processes; it can be interpreted through techniques of non-associative algebras. The bridge between these two fields is established and explored by Tian in the only book of this beautiful subject, \cite{tian}, where the author formulates theorems of Markov chain theory in the context of evolution algebras and related operators. The book also includes a review of examples and applications, as well as different open problems in this area of research. One of the open questions is what is the relationship between the evolution algebra induced by a random walk on a graph, which is a special type of Markov chain, and the evolution algebra determined by the same graph. The purpose of our work is to answer this question by looking for the existence of isomorphisms between these structures. As far as we know, this question has not been addressed yet (see \cite{tian2}). Our results cover a wide range of finite and infinite graphs, including the families of finite graphs that were recently considered by \cite{nunez/2014}. Indeed, we point out that this work provides a contribution to the theory initiated by \cite{nunez/2013,nunez/2014}. 

\smallskip
The paper is organized as follows. Section 2 is devoted to preliminary definitions and examples. Section 3 includes the main results of our work, which are subdivided into three parts. The first part is related to the existence of isomorphisms between the evolution algebras of interest, when the underlying graph is a regular or a complete bipartite graph. In the second part we show some examples of graphs where the only homomorphism is the null map. This is the case of most path, friendship, and wheel graphs. Finally we discuss the case of complete $n$-partite graphs, which may be an interesting issue of further research.

\section{Evolution algebras, random walks and graphs}

At first we consider the definition of evolution algebra associated to a graph introduced by \cite{tian}, and studied recently by \cite{nunez/2013,nunez/2014}. Then, we consider the evolution algebra of a random walk on a graph. As the random walk is a special type of discrete time Markov chain the induced algebra is just the associated Markov evolution algebra.  

\subsection{Evolution algebra of a graph}  

Lets start with some notation regarding graph theory. We refer the reader to \cite{Bondy} for a general reference on Graph Theory. A graph $G$ with $n$ vertices ($n$ may be infinite) is a pair $(V,E)$ where $V:=\{1,\ldots,n\}$ is the set of vertices and $E:=\{(i,j)\in V\times V:i\leq j\}$ is the set of edges. If $(i,j)\in E$ we say that $i$ and $j$ are neighbors. In the notation above we assume $i\leq j$ for the sake of simplicity; this means that we are considering undirected, or simple, graphs and the existence of loops (i.e., if $i=j$). However, we point out that our results may be extended to directed graphs without further work. In addition, we let $A=(a_{ij})$ the adjacency matrix of $G$, i.e. 
\[   
a_{ij}=\left\{
\begin{array}{ll}
1,&\text{if }(i,j)\in E \text{ or }(j,i)\in E,\\[.2cm]
0,&\text{other case.}
\end{array}\right.
\]

As we consider undirected graphs we have $a_{ij}=a_{ji}$, for $i,j\in V$. Note that two vertices $i$ and $j$ are neighbors if $a_{ij}=1$. In what follows we shall consider locally finite graphs, i.e., the number of neighbors of any vertex is finite. 
This assumption is important when considering the random walk on such graph.

\smallskip
The evolution algebra of a graph $G$ is defined by \cite[Section 6.1]{tian} as follows.

\smallskip
\begin{defn}\label{def:eagraph}
Let $G=(V,E)$ a graph with adjacency matrix given by $A=(a_{ij})$. The evolution algebra associated to $G$ is the evolution algebra $\A(G)$ with natural basis $S=\{e_i: i\in V\}$, and relations

\[
 \begin{array}{ll}\displaystyle
e_i \cdot e_i = \sum_{k\in V}^{n} a_{ik} e_k,&\text{for  }i \in  V,\\[.3cm]
\end{array}
\] 
\noindent
and $e_i \cdot e_j =0,\text{if }i\neq j.$
\end{defn}

\smallskip
\begin{exa} \label{ex:complete} The complete graph is an undirected graph in which every pair of different vertices is connected by a unique edge (see Figure \ref{FIG:complete}). 

\begin{figure}[h]
\label{FIG:complete}
\begin{center}
\begin{tikzpicture}[scale=0.4]
\GraphInit[vstyle=Simple]
    \SetGraphUnit{1}
    \tikzset{VertexStyle/.style = {shape = circle,fill = black,minimum size = 2pt,inner sep=1.5pt}}
  \begin{scope}[xshift=12cm]
  \grComplete[RA=4.5]{8}
  \end{scope}
\end{tikzpicture}
\end{center}
\caption{Complete graph with $8$ vertices, $K_8$.}
\end{figure}
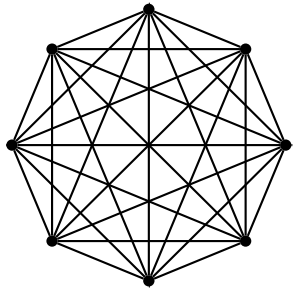

\noindent
Let $K_n$ be the complete graph with $n$ vertices. Then $\mathcal{A}(K_n)$ is the  evolution algebra with set of generators $\{e_1, e_{2}, \ldots,  e_{n}\}$ and relations 
\[
\begin{array}{ll}\displaystyle
 e_i^2=\sum_{\substack{j=1,\\j \neq i}}^n e_j,& \text{for }i\in\{1,2,\ldots,n\},\\[.2cm]
\end{array}
\]
and $e_i \cdot e_j =0,$  for $i\neq j.$
\end{exa}

\smallskip
\begin{exa} \label{ex:tree} The $d$-dimensional homogeneous tree $\mathbb{T}_d$ is an undirected infinite graph in which every vertex has degree $d+1$, and every pair of different vertices is connected by a unique path, i.e., a sequence of neighbors vertices (see Figure \ref{FIG:tree}(a)).
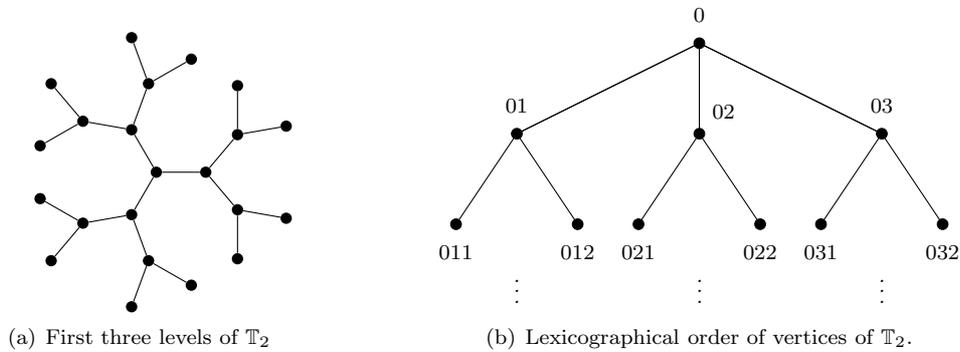
\begin{figure}[h]
\label{FIG:tree}
\begin{center}

\subfigure[][First three levels of $\mathbb{T}_{2}$]{

\tikzstyle{level 1}=[sibling angle=120]
\tikzstyle{level 2}=[sibling angle=100]
\tikzstyle{level 3}=[sibling angle=80]
\tikzstyle{every node}=[fill]

\tikzstyle{edge from parent}=[segment length=0.8mm,segment angle=10,draw]
\begin{tikzpicture}[scale=0.5,grow cyclic,shape=circle,minimum size = 2pt,inner sep=1.5pt,level distance=13mm,
                    cap=round]
\node {} child [\A] foreach \A in {black,black,black}
    { node {} child [color=\A!50!\B] foreach \B in {black,black}
        { node {} child [color=\A!50!\B!50!\C] foreach \C in {black,black}
            { node {} }
        }
    };
    \end{tikzpicture}
}
\qquad \qquad
\subfigure[][Lexicographical order of vertices of $\mathbb{T}_2$.]{

\begin{tikzpicture}[scale=0.8]

\draw (0,0) -- (-3,-1.5)--(-4,-3);
\draw (0,0) -- (-3,-1.5)--(-2,-3);
\draw (0,0) -- (0,-1.5)--(-1,-3);
\draw (0,0) -- (0,-1.5)--(1,-3);
\draw (0,0) -- (3,-1.5)--(4,-3);
\draw (0,0) -- (3,-1.5)--(2,-3);

\draw (-3,-3.5) node[below,font=\footnotesize] {$\vdots$};
\draw (0,-3.5) node[below,font=\footnotesize] {$\vdots$};
\draw (3,-3.5) node[below,font=\footnotesize] {$\vdots$};


\filldraw [black] (0,0) circle (2.5pt);
\draw (0,0.2) node[above,font=\footnotesize] {$0$};
\filldraw [black] (0,-1.5) circle (2.5pt);
\draw (-3,-1.3) node[above,font=\footnotesize] {${01}$};
\filldraw [black] (3,-1.5) circle (2.5pt);
\draw (0.4,-1.4) node[above,font=\footnotesize] {${02}$};
\filldraw [black] (-3,-1.5) circle (2.5pt);
\draw (3,-1.3) node[above,font=\footnotesize] {${03}$};

\filldraw [black] (-4,-3) circle (2.5pt);
\draw (-4,-3.2) node[below,font=\footnotesize] {${011}$};
\filldraw [black] (-2,-3) circle (2.5pt);
\draw (-2,-3.2) node[below,font=\footnotesize] {${012}$};
\filldraw [black] (-1,-3) circle (2.5pt);
\draw (-1,-3.2) node[below,font=\footnotesize] {${021}$};
\filldraw [black] (1,-3) circle (2.5pt);
\draw (1,-3.2) node[below,font=\footnotesize] {${022}$};
\filldraw [black] (2,-3) circle (2.5pt);
\draw (2,-3.2) node[below,font=\footnotesize] {${031}$};
\filldraw [black] (4,-3) circle (2.5pt);
\draw (4,-3.2) node[below,font=\footnotesize] {${032}$};

\end{tikzpicture}
}

\end{center}
\caption{$2$-dimensional homogeneous tree $\mathbb{T}_2$.}
\end{figure}

\noindent
In this case, it is more fruitful to use the lexicographical order to label the vertices: we use $0$ for a vertex usually identified as the root of the tree, and we imagine the tree as growing upwards away from its root. We let $0\,1,0\,2,\ldots,0\, (d+1)$ those vertices connected through an edge to the root; $011,012,\ldots, 01d$ are the vertices connected to the vertex $01$, which  are further from the root, and so on (see Figure \ref{FIG:tree}(b)). Then $\mathcal{A}(\mathbb{T}_d)$ is the evolution algebra with the infinite set of generators $\{e_0, e_{0\,1},e_{0\,2},\ldots,e_{0\,(d+1)},e_{011},e_{012},\ldots,e_{01d},e_{021}, \ldots \}$ and relations:
\[
\begin{array}{rcl}\displaystyle
e_0^2 & =&\displaystyle \sum_{j=1}^{d+1} e_{0j},\\[.3cm]
 e_{0\,i_1\,i_2\,\ldots \,i_{k+1}}^2&=&e_{0\, i_1\,i_2\,\ldots \,i_{k}} +\displaystyle \sum_{\substack{j=1}}^{d} e_{0\, i_1\,i_2\,\ldots \,i_{k+1}\,j}, \text{ for }k\in \mathbb{Z}^{+}\\[.3cm]
\end{array}
\]
where $e_{0\,i_1\,i_2\,\ldots \,i_{k}} :=e_{0}$ whether $k=0$, and $e_{\sigma} \cdot e_\nu =0,$ such that $\sigma\neq \nu.$
\end{exa}

\smallskip
We refer the reader to \cite{nunez/2013,nunez/2014} for a review of evolution algebras associated to well-known families of finite graphs. We point out that almost every evolution algebra included in our work has been previously collected in \cite{nunez/2013,nunez/2014}. 


\subsection{Evolution algebra of a random walk on a graph}

The random walk on $G=(V,E)$ is a discrete time Markov chain $(X_n)_{n\geq 0}$ with state space given by $V$ and transition probabilities given by
$$p_{ik}=\frac{a_{ik}}{d_i},$$
where $i,k\in V$ and
$$d_i:=\sum_{k\in V} a_{ik},$$
is the number of neighbors of vertex $i$. In other words, the sequence of random variables $(X_n)_{n\geq 0}$ denotes the set of positions of a particle walking around the vertices of $G$, where each new position is selected at random from the set of neighbors of the current position. From this definition of random walk, we can consider its related Markov evolution algebra.

\smallskip
\begin{defn}
Let $G=(V,E)$ be a graph with adjacency matrix given by $A=(a_{ij})$. We define the evolution algebra associated to the random walk on $G$ as the evolution algebra $\A_{RW}(G)$ with natural basis $S=\{e_i: i\in V\}$, and relations given by
\[
\begin{array}{ll}\displaystyle
e_i \cdot e_i = \sum_{k\in V}\left( \frac{a_{ik}}{d_i}\right)e_k,&\text{for }i  \in V,
\end{array}
\] 
\noindent
and $e_i \cdot e_j =0, \text{ if } i\neq j.$
\end{defn}

\begin{exa} Let $K_n$ be the complete graph with $n$ vertices considered in Example \ref{ex:complete}. Then $\mathcal{A}_{RW}(K_n)$ is the evolution algebra with set of generators $\{e_1,e_2,\ldots, e_n\}$ and relations:
\[
\begin{array}{ll}\displaystyle
 e_i^2=\frac{1}{n-1} \,\sum_{\substack{j=1,\\j \neq i}}^n e_j,& \text{ for }i\in\{1,2,\ldots,n\},\\[.2cm]
\end{array}
\]
and $e_i \cdot e_j =0,$  for $i\neq j.$
\end{exa} 

\begin{exa}
Let $\mathbb{T}_d$ be the $d$-dimensional homogeneous tree considered in Example \ref{ex:tree}. Then $\mathcal{A}_{RW}(\mathbb{T}_d)$ is the evolution algebra with set of generators $\{e_0, e_{0\,1},\ldots,e_{0\,(d+1)},e_{011},\ldots,e_{01d},e_{021}, \ldots \}$ and relations:
\[
\begin{array}{rcl}\displaystyle
e_0^2 & =&\displaystyle\frac{1}{d+1}\left(\displaystyle \sum_{j=1}^{d+1} e_{0j}\right),\\[.5cm]
 e_{0\,i_1\,i_2\,\ldots \,i_{k+1}}^2&=&\displaystyle\frac{1}{d+1}\left(e_{0\, i_1\,i_2\,\ldots \,i_{k}} + \displaystyle\sum_{\substack{j=1}}^{d} e_{0\, i_1\,i_2\,\ldots \,i_{k+1}\,j}\right), \text{ for }k\in \mathbb{Z}^{+}\\[.2cm]
\end{array}
\]
where $e_{0\,i_1\,i_2\,\ldots \,i_{k}} :=e_{0}$ whether $k=0$, and $e_{\sigma} \cdot e_\nu =0,$ such that $\sigma\neq \nu.$

\end{exa}

\section{On the existence of isomorphisms between $\A(G)$ and $\A_{RW}(G)$}

The purpose of this work is to explore the connection between the algebras $\A(G)$ and $\A_{RW}(G)$, for a given graph $G$. As mentioned in the Introduction, this is one of the open problems stated in \cite[Chapter 6]{tian}, and more recently in \cite{tian2}. In order to do it, we consider the following definition given by \cite[Section 3.1]{tian}. 

\smallskip
\begin{defn}
Let $\mathcal{A}$ and $\tilde{\mathcal{A}}$ be $\mathbb{K}$-evolution algebras and $S=\{e_1,e_2,\ldots,e_n,\ldots\}$ a natural basis for $\mathcal{A}$. We say that a $\mathbb{K}$-linear map $g: \mathcal{A} \longrightarrow \tilde{\mathcal{A}}$ is an homomorphism of evolution algebras if it is an algebraic map and if the set $\{g(e_1),\ldots,g(e_n)\}$ can be complemented to a natural basis of  $\tilde{\mathcal{A}}$. In addition, if $g$ is bijective, then we say that it is an isomorphism.
\end{defn}

In the previous definition, $g$ is an algebraic map in the sense that $g$ preserves products; i.e., for any $e_i, e_j \in S$ we have $g(e_i \cdot e_j)= g(e_i) \cdot g(e_j)$.

\subsection{Regular and complete bipartite graphs}

Our first result states that the evolution algebra induced by the random walk on a graph and the evolution algebra determined by the same graph are isomorphic as evolution algebras provided the graph is well-behaved in some sense. We shall consider first the case of regular graphs, i.e., any vertex has exactly the same number of neighbors. Notice that a complete graph and an  homogeneous tree are examples of regular graphs, see Examples \ref{ex:complete} and \ref{ex:tree}. Next we analyze the case of a complete bipartite graph $K_{m,n}$, where the set of vertices  can be partitioned into two subsets, of sizes $m$ and $n$, such that there is no edge connecting two vertices in the same subset, and every possible edge that could connect vertices in different subsets is part of the graph, see Figure \ref{fig:bcomplete}.

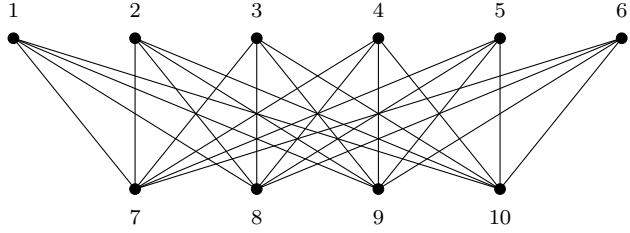
\begin{figure}[h]\label{fig:bcomplete}
\begin{center}
\begin{tikzpicture}[scale=0.8]

\draw (-5,0) -- (-3,-2.5);
\draw (-5,0) -- (-1,-2.5);
\draw (-5,0) -- (1,-2.5);
\draw (-5,0) -- (3,-2.5);

\draw (-3,0) -- (-3,-2.5);
\draw (-3,0) -- (-1,-2.5);
\draw (-3,0) -- (1,-2.5);
\draw (-3,0) -- (3,-2.5);

\draw (-1,0) -- (-3,-2.5);
\draw (-1,0) -- (-1,-2.5);
\draw (-1,0) -- (1,-2.5);
\draw (-1,0) -- (3,-2.5);

\draw (1,0) -- (-3,-2.5);
\draw (1,0) -- (-1,-2.5);
\draw (1,0) -- (1,-2.5);
\draw (1,0) -- (3,-2.5);

\draw (3,0) -- (-3,-2.5);
\draw (3,0) -- (-1,-2.5);
\draw (3,0) -- (1,-2.5);
\draw (3,0) -- (3,-2.5);

\draw (5,0) -- (-3,-2.5);
\draw (5,0) -- (-1,-2.5);
\draw (5,0) -- (1,-2.5);
\draw (5,0) -- (3,-2.5);

\filldraw [black] (-5,0) circle (2.5pt);
\draw (-5,0.2) node[above,font=\footnotesize] {$1$};
\filldraw [black] (-3,0) circle (2.5pt);
\draw (-3,0.2) node[above,font=\footnotesize] {$2$};
\filldraw [black] (-1,0) circle (2.5pt);
\draw (-1,0.2) node[above,font=\footnotesize] {$3$};
\filldraw [black] (1,0) circle (2.5pt);
\draw (1,0.2) node[above,font=\footnotesize] {$4$};
\filldraw [black] (3,0) circle (2.5pt);
\draw (3,0.2) node[above,font=\footnotesize] {$5$};
\filldraw [black] (5,0) circle (2.5pt);
\draw (5,0.2) node[above,font=\footnotesize] {$6$};

\filldraw [black] (-3,-2.5) circle (2.5pt);
\draw (-3,-2.7) node[below,font=\footnotesize] {$7$};
\filldraw [black] (-1,-2.5) circle (2.5pt);
\draw (-1,-2.7) node[below,font=\footnotesize] {$8$};
\filldraw [black] (1,-2.5) circle (2.5pt);
\draw (1,-2.7) node[below,font=\footnotesize] {$9$};
\filldraw [black] (3,-2.5) circle (2.5pt);
\draw (3,-2.7) node[below,font=\footnotesize] {$10$};

\end{tikzpicture}
\caption{Complete bipartite graph $K_{6,4}$.}
\end{center}
\end{figure}

\smallskip
\begin{theorem}\label{teo:iso}
$\A(G)$ and $\A_{RW}(G)$ are isomorphic as evolution algebras in the following cases.

\begin{enumerate}[i.]
\item $G=G_d$ is a $d$-regular graph, for $d\geq 1$;
\item $G=K_{m,n}$ is the complete bipartite graph with partitions of sizes $m$ and $n$, for $m,n\geq 1$.

\end{enumerate}
\end{theorem}

\begin{proof}
\smallskip
\noindent
{\it i.} Assume that $G_d=(V,E)$ is a $d$-regular graph, i.e. $d_i=d$, for any $i\in V$. The induced evolution algebras $\A(G_d)$ and $\A_{RW}(G_d)$ are obtained by considering the set of generators $\{e_i, i\in V\}$ and relations:
\smallskip
$$
\mathcal{A}(G_d):\left\{
\begin{array}{cl}
 e_i^2=\displaystyle \sum_{j\in V}a_{ij} e_{j}, &  \text{ for } i\in V,\\[.2cm]
 e_i \cdot e_j =0, & \text{ for } i\neq j,
\end{array}\right.
$$
and 
\smallskip
$$
\mathcal{A}_{RW}(G_d):\left\{
\begin{array}{cl}
 e_i^2= \displaystyle \sum_{j\in V}\frac{a_{ij}}{d} e_{j}, &  \text{ for } i\in V,\\[.2cm]
 e_i \cdot e_j =0,& \text{ for }i\neq j.
\end{array}\right.
$$

Consider the $\mathbb{R}$-linear transformation $g:\A(G)\longrightarrow \A_{RW}(G)$ defined by $g(e_i)=d \, e_i$ for $i\in V$. Thus defined it is not difficult to see that $g$ is an evolution homomorphism and, since $g$ send a basis of $\A(G_d)$ into a basis of $\A_{RW}(G_d)$, it is an evolution isomorphism.


\smallskip
\noindent
{\it ii.} Let $G=K_{m,n}$, for $m,n\geq 1$  be a complete bipartite graph with partitions of sizes $m$ and $n$. In other words, the set of vertices of $K_{m,n}$ can be partitioned into two subsets, say $V_1:=\{1,\ldots,m\}$ and $V_2:=\{m+1,\ldots,m+n\}$, such that there is no edge connecting two vertices in the same subset, and every possible edge that could connect vertices in different subsets is part of the graph. The resulting evolution algebras associated to $K_{m,n}$ are given by the set of generators $\{e_1,\ldots,e_m,e_{m+1},\ldots e_{m+n}\}$ and relations:
\smallskip
$$
\mathcal{A}(K_{m,n}):\left\{
\begin{array}{cl}
 e_i^2=\displaystyle \sum_{j=1}^n e_{m+j}, &  \text{ for } i\in\{1,2,\ldots,m\},\\[.5cm]
  e_i^2=\displaystyle \sum_{j=1}^m e_{j}, & \text{ for } i\in \{m+1,m+2,\ldots,m+n\},\\[.5cm]
 e_i \cdot e_j =0, & \text{ for } i\neq j,
\end{array}\right.
$$
and 
\smallskip
$$
\mathcal{A}_{RW}(K_{m,n}):\left\{
\begin{array}{cl}
 e_i^2= \displaystyle\sum_{j=1}^n \frac{1}{n}e_{m+j}, & \text{ for } i\in\{1,2,\ldots,m\},\\[.5cm]
 e_i^2=\displaystyle \sum_{j=1}^m \frac{1}{m}e_{j}, & \text{ for } i\in \{m+1,m+2,\ldots,m+n\},\\[.5cm]
 e_i \cdot e_j =0,& \text{ for }i\neq j.
\end{array}\right.
$$

\smallskip
Let $g:\A(K_{m,n})\longrightarrow \A_{RW}(K_{m,n})$ be a $\mathbb{R}$-linear transformation defined by
$$
g(e_i)=\left\{
\begin{array}{ll}
m^{1/3}n^{2/3} e_i, &\text{for }i\in\{1,2,\ldots,m\};\\[.2cm]
m^{2/3}n^{1/3} e_i, &\text{for }i\in \{m+1,m+2,\ldots,m+n\}.
\end{array}\right.
$$
It is not difficult to see that $g$ is an evolution homomorphism and, since $g$ send a basis of $\A(K_{m,n})$ into a basis of $\A_{RW}(K_{m,n})$, it is an evolution isomorphism.

\end{proof}

Theorem \ref{teo:iso} holds for any regular graph, finite or infinite, including snark and Petersen graphs whose evolution algebras where introduced by \cite{nunez/2014}. Theorem \ref{teo:iso} is also true whenever we consider a complete bipartite graph, like a star graph ($K_{1,n}$) or a utility graph ($K_{3,3}$). We point out that with this result one can study the evolution algebra of a random walk on a graph in order to understand the evolution algebra of the same graph, provided the graph belongs to one of the families covered by Theorem \ref{teo:iso}. One can accomplish this task by combining well known results of random walks on graphs together with the connections between evolution algebras and Markov chains stated by \cite[Chapter 4]{tian}.

\subsection{Path, friendship and wheel graphs}

In this section we list some graphs for which $\mathcal{A}(G)$ and $\mathcal{A}_{RW}(G)$ are not isomorphic as evolution algebras. Further, our results are stronger in the sense that we shall prove that the only evolution homomorphism between these algebras is the null map. 



Let us consider as underlying graph the path graph with $n$ vertices, denoted by $A_n$, where each vertex $i$ is connected to $i-1$ and $i+1$, for $i\in \{2,\ldots,n-1\}$, and $1$ is connected only to $2$ while $n$ is connected only to $n-1$ (see Figure \ref{fig:interval}).

\smallskip
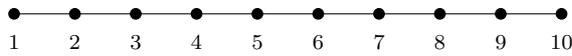
\begin{figure}[h]
\label{FIG:FP}
\begin{center}
\begin{tikzpicture}[scale=0.8]

\draw (-3,-2) -- (6,-2);
\draw (-3,-2.2) node[below,font=\footnotesize] {$1$};


\draw [very thick] (-3,-2) circle (2pt);
\filldraw [black] (-3,-2) circle (2pt);
\draw (-2,-2.2) node[below,font=\footnotesize] {$2$};
\draw [very thick] (-2,-2) circle (2pt);
\filldraw [black] (-2,-2) circle (2pt);
\draw [very thick] (-1,-2) circle (2pt);
\filldraw [black] (-1,-2) circle (2pt);
\draw (-1,-2.2) node[below,font=\footnotesize] {$3$};
\draw [very thick] (0,-2) circle (2pt);
\filldraw [black] (0,-2) circle (2pt);
\draw (0,-2.2) node[below,font=\footnotesize] {$4$};
\draw [very thick] (1,-2) circle (2pt);
\filldraw [black] (1,-2) circle (2pt);
\draw (1,-2.2) node[below,font=\footnotesize] {$5$};
\draw [very thick] (2,-2) circle (2pt);
\filldraw [black] (2,-2) circle (2pt);
\draw (2,-2.2) node[below,font=\footnotesize] {$6$};
\draw [very thick] (3,-2) circle (2pt);
\filldraw [black] (3,-2) circle (2pt);
\draw (3,-2.2) node[below,font=\footnotesize] {$7$};
\draw [very thick] (4,-2) circle (2pt);
\filldraw [black] (4,-2) circle (2pt);
\draw (4,-2.2) node[below,font=\footnotesize] {$8$};
\draw [very thick] (5,-2) circle (2pt);
\filldraw [black] (5,-2) circle (2pt);
\draw (5,-2.2) node[below,font=\footnotesize] {$9$};
\draw [very thick] (6,-2) circle (2pt);
\filldraw [black] (6,-2) circle (2pt);
\draw (6,-2.2) node[below,font=\footnotesize] {$10$};

\end{tikzpicture}
\end{center}
\label{fig:interval}
\caption{Path graph $A_{10}$.}
\end{figure}

As a consequence of Theorem \ref{teo:iso}(ii) we have that  $\A(A_3) \cong \A_{RW}(A_3)$ as evolution algebras. Indeed, notice that  $A_3 \cong K_{1,2}$ as graphs. We shall show that $\A(A_n) \ncong \A_{RW}(A_n)$ as evolution algebras, for $n>3$.   

\smallskip
\begin{prop}\label{prop:path}
Let  $A_n$ be a path graph, with $n>3$. Then, the only evolution homomorphism between $\A(A_n)$ and $\A_{RW}(A_n)$ is the null map. In particular, $\A(A_n) \ncong \A_{RW}(A_n)$ as evolution algebras.
\end{prop}

\begin{proof}
Consider the evolution algebras induced by $A_n$, and by the random walk on $A_n$, respectively. That is, consider the evolution algebras whose set of generators is $\{e_1,e_2,\ldots,e_n\}$ and relations are:

\smallskip
$$
\mathcal{A}(A_n): \left\{
\begin{array}{cl}
 e_1^2= e_2,&\\[.2cm]
  e_i^2=e_{i-1}+e_{i+1}, & \text{ for } i\in \{2,\ldots,n-1\},\\[.2cm]
 e_n^2=e_{n-1}, &\\[.2cm]
 e_i \cdot e_j =0, & \text{ for }i\neq j,
\end{array}\right.
$$
\smallskip
and

\smallskip
$$
\mathcal{A}_{RW}(A_n):\left\{
\begin{array}{cl}
 e_1^2= e_2,&\\[.2cm]
 e_i^2=\frac{1}{2}e_{i-1}+\frac{1}{2}e_{i+1}, & \text{ for } i\in \{2,\ldots,n-1\},\\[.2cm]
 e_n^2=e_{n-1}, &\\[.2cm]
 e_i \cdot e_j =0,&  \text{ for }i\neq j. 
\end{array}\right.
$$

\smallskip

Now assume that there exists an evolution homomorphism $g:\mathcal{A}(A_n)\longrightarrow \mathcal{A}_{RW}(A_n)$, such that 
\begin{equation}
\nonumber g(e_i)=\sum_{k=1}^{n} t_{ik}e_k, \,\,\, \text{for any } i\in \{1,2,\ldots,n\},
\end{equation}
where the $t_{ik}$'s are scalars.  Thus 
\begin{eqnarray} \label{coe1}
\nonumber g(e_i)\cdot g(e_j) &=&  \sum_{k=1}^{n} t_{ik}t_{jk}e_{k}^{2}  \\
\nonumber                             &=& \left( \frac {t_{i2}t_{j2} }{2} \right) e_1 +  \left(  t_{i1}t_{j1} + \frac {t_{i3}t_{j3} }{2} \right) e_2 +  \sum_{k=3}^{n-2}\left(\frac{ t_{i(k-1)}t_{j(k-1)}  + t_{i(k+1)}t_{j(k+1)} }{2}  \right) e_k  +\\
 & &  \left(\frac{ t_{i(n-2)}t_{j(n-2)} }{2}+ {t_{in}t_{jn} } \right)e_{n-1} +  \left( \frac {t_{i(n-1)}t_{j(n-1)} }{2} \right)e_n.
\end{eqnarray}

Since $g(e_i)\cdot g(e_j)=0$ for any $i\neq j$, we obtain by \eqref{coe1} the following set of equalities:

\begin{eqnarray}
                      t_{i2}t_{j2}  &=& 0,\label{eq:nova3}\\ 
 t_{i1}t_{j1} + \frac {t_{i3}t_{j3} }{2} &=& 0,\label{eq:nsqp6}\\
  t_{i(k-1)}t_{j(k-1)}  + t_{i(k+1)}t_{j(k+1)} & = &0, \text{ for }  k \in \{3, \ldots, n-2\} \label{eq:nsqp3}\\
 \frac{ t_{i(n-2)}t_{j(n-2)} }{2}+ {t_{in}t_{jn} }  &=& 0,\label{eq:nsqp4}\\
 t_{i(n-1)}t_{j(n-1)} &=& 0.\label{eq:nova7}
\end{eqnarray}

First we notice that whether $n$ is even the equations \eqref{eq:nova3}, \eqref{eq:nsqp3} (for $k$ odd) and \eqref{eq:nsqp4} allow us to conclude that $t_{i\ell}t_{j\ell}=0$ whether $i\neq j$ and $\ell$ is even, while the equations \eqref{eq:nsqp6}, \eqref{eq:nsqp3} (for $k$ even) and \eqref{eq:nova7} imply $t_{i\ell}t_{j\ell}=0$ whether $i\neq j$ and $\ell$ is odd. Therefore, we obtain
\begin{equation}\label{eq:tij}
t_{i\ell}t_{j\ell}=0, \text{ for any }i,j,\ell\in \{1,2,\ldots,n\} \text{ with }i\neq j.
\end{equation}

In order to show \eqref{eq:tij} for the case in which $n$ is odd we need a bit more of work. At first we obtain, again by \eqref{eq:nova3}, \eqref{eq:nsqp3} (for $k$ odd) and \eqref{eq:nsqp4} that
\begin{eqnarray}
                   t_{i\ell}t_{j\ell} &= &0, \text{ if } \ell  \text{ is even} \label{eq:ynqp}.
 \end{eqnarray}

 Our purpose from now is to show that \eqref{eq:ynqp} holds also when $\ell$ is odd. We shall accomplish this by showing that $t_{i1}t_{j1}=0$ for $i\neq j$ and applying \eqref{eq:nsqp6},  \eqref{eq:nsqp3} (for $k$ even) and \eqref{eq:nsqp4}. First observe that
\begin{eqnarray}
\label{coe2} g(e_{1}^{2} )  &=&  g(e_{2}) = \sum_{k=1}^{n} t_{2k}e_k, \\
\label{coe4}  g(e_{i}^2)      &=& g(e_{i-1})+g(e_{i+1}) = \sum_{k=1}^{n} (t_{(i-1)k} + t_{(i+1)k}  )e_{k}, \text{ for } i \in \{ 2, \ldots, n-1\},\\
\label{coe3} g(e_{n}^{2} )  &=&  g(e_{n-1}) =  \sum_{k=1}^{n} t_{(n-1)k}e_{k}.
\end{eqnarray}

It will be useful the set of identities obtained from \eqref{coe1} and \eqref{coe2}-\eqref{coe3} by observing the second coefficient of $g(e_j^2)$ for any $j$, i.e.:

\begin{equation}\label{eq:novanova}
t_{j1}^{2} + \frac{t_{j3}^2}{2}=\left\{
\begin{array}{ll}
 t_{22},&  \text{ for } j =1,  \\[.2cm]
 t_{(j-1)2} +   t_{(j+1)2}, & \text{ for } j \in \{2,\ldots,n-1\},\\[.2cm]
 t_{(n-1)2}, & \text{ for } j=n.  \\
\end{array}\right.
\end{equation}

We notice that having $t_{i2}=0$ for any $i$ implies by \eqref{eq:novanova} that $t_{i1}=0$ for any $i$. Lets assume $t_{i_{0}2} \not= 0$ for some $i_{0}\in \{1, \ldots, n\}$. Then, we have that 
\begin{equation}\label{zeros}
t_{k2}=0, \text{ for }  k \not= i_{0}.
\end{equation}

Moreover, by \eqref{eq:novanova} the second coefficient of $g(e_{i_{0}}^2)$ satisfies

\begin{equation}\label{eq:nsqp5}
t_{i_{0}1}^{2} + \frac{t_{i_{0}3}^2}{2}=\left\{
\begin{array}{ll}
 t_{22},&  \text{ for } i_{0}=1,  \\[.2cm]
 t_{(i_{0}-1)2} +   t_{(i_{0}+1)2}, & \text{ for } i_{0}\in \{2,\ldots,n-1\},\\[.2cm]
 t_{(n-1)2}, & \text{ for } i_{0}=n.  \\
\end{array}\right.
\end{equation}

Observe that the equation above implies $t_{i_{0}1}=0$. Now, consider $j$ in such a way that $j$ is not a neighbor of $i_0$ (of course $j \not=i_{0}$). By  (\ref{zeros}) and \eqref{eq:nsqp5} we have that

\begin{equation} \label{noveci}
t_{j1}=0,\text{  if } j  \text{ is not a neighbor of } i_{0}.
\end{equation}
\noindent
From now on we consider four cases, according to the possible values of $i_0$:\\

\smallskip
\noindent
{\bf \underline{Case 1}: $i_{0}=1$.} Since $t_{11}=0$ and by \eqref{noveci} we have $t_{k1}=0$ for $k \not =2.$ Therefore $t_{i1}t_{j1}=0 \text{ for } i \not= j.$

\smallskip
\noindent
{\bf \underline{Case 2}: $i_{0}=n$.} As before $t_{k1}=0$ for  $k \not = n-1.$ Therefore
$t_{i1}t_{j1}=0 \text{ for } i \not= j.$

\smallskip
\noindent
{\bf \underline{Case 3}: $i_{0} \notin \{1, n-2,n-1,n\}$.} By (\ref{noveci}) we have that  $t_{k1}=0$,  for  $k \notin\{i_{0} -1, i_{0} +1\}$. Let  $j:=i_{0} +2$. By (\ref{coe1})  and (\ref{coe4}),    the first coefficient  of $g(e_{j}^2)$ satisfies
$$
\frac {t^{2}_{(i_{0}+2)2}}{2}=t_{(i_{0}+1)1} + t_{(i_{0}+3)1}.
$$
But $t_{(i_{0}+2)2}=0$ by  (\ref{zeros}),  and $t_{(i_{0}+3)1}=0$ by  (\ref{noveci}). Hence $t_{(i_{0}+1)1}=0$, which implies that $t_{k1}=0$ for $k\neq i_0 -1$, and therefore
$t_{i1}t_{j1}=0, \text{ for } i\not= j.$

\smallskip
\noindent
{\bf \underline{Case 4}: $i_{0}\in \{n-1, n-2\}$.}  Let $j:=i_{0}-2$. By (\ref{coe1})  and (\ref{coe4}),   the first coefficient  of $g(e_{j}^2)$ satisfies
$$
\frac {t^{2}_{(i_{0}-2)2}}{2}= t_{(i_{0}-3)1}+t_{(i_{0}-1)1} .
$$
But $t_{(i_{0}-2)2}=0$ by  (\ref{zeros}), and $t_{(i_{0}-3)1}=0$ by (\ref{noveci}). Then $t_{(i_{0}-1)1}=0$, and therefore  $t_{k1}=0$ for $k\neq i_0 +1$. Thus $t_{i1}t_{j1}=0, \text{ for } i\not= j.$ Note that whether $n=5$ and $ i_{0}=n-2$, so $i_0=3$, then $j=1$. In this case the first coefficient of $g(e_j^2)$ satisfies $t_{12 }^{2}= t_{21}.$ But $t_{12}=0$ by  (\ref{zeros}) and then $t_{21}=0$. So the same conclusion is valid for this case.

\smallskip
Our previous analysis for the case in which $n$ is odd is enough to verify that \eqref{eq:tij} indeed holds for any value of $n>3$. Indeed, we known \eqref{eq:ynqp} (for $\ell$ even), we proved $t_{i1}t_{j1}=0$ for $i\neq j$ and we applied \eqref{eq:nsqp6},  \eqref{eq:nsqp3} (for $k$ even) and \eqref{eq:nsqp4} to get $t_{i\ell}t_{j\ell}=0$ for $\ell$ odd. Therefore \eqref{eq:ynqp} holds for any value of $\ell$. This in turns implies that  for any $k$, $t_{ik}\neq 0$ for at most one of the values of $i$. In other words, if the map $g$ exists, then it must be defined as
$$g(e_i) = \alpha_i e_{\pi(i)},$$
for $i\in \{1,2,\ldots,n\}$, where the $\alpha_i's$ are scalars and $\pi$ is an element of the symmetric group $S_n$. Again, since $g$ is an evolution homomorphism we have $g(e_i^2)=g(e_i)\cdot g(e_i)$, for any $i$. In particular,  $g(e_1^2)=g(e_1)\cdot g(e_1)$; it follows that 
$$g(e_1^2) = g(e_2)=\alpha_2 e_{\pi(2)},$$
is equal to
$$g(e_1)\cdot g(e_1) = \alpha_1^2 e_{\pi(1)}^2.$$

\noindent Hence  $\alpha_2 = \alpha_1^2$, and $\pi(1)\in \{1,n\}$. Analogously,  $g(e_n^2)=g(e_n)\cdot g(e_n)$ and then
$$g(e_n^2) = g(e_{n-1})=\alpha_{n-1} e_{\pi(n-1)},$$
is equal to
$$g(e_n)\cdot g(e_n) = \alpha_{n}^2 e_{\pi(n)}^2,$$
and we obtain $\alpha_{n-1}=\alpha_{n}^2$, and $\pi(n)\in\{1,n\}$. For $i\in \{2, \ldots, n-1\}$ ($\pi(i)\notin \{1,n\}$) we have on one hand 
$$g(e_i)\cdot g(e_i)=\alpha_{i}^2 e_{\pi(i)}^2 = \frac{\alpha_{i}^2}{2}( e_{\pi(i)-1}+e_{\pi(i)+1}), $$
and, on the other hand
$$g(e_i^2)=g(e_{i-1}+e_{i+1})= \alpha_{i-1} e_{\pi(i-1)}+ \alpha_{i+1} e_{\pi(i+1)}.$$
Thus, $g(e_i^2)=g(e_i)\cdot g(e_i)$ implies
$$\frac{\alpha_{i}^2}{2}( e_{\pi(i)-1}+e_{\pi(i)+1}) =  \alpha_{i-1} e_{\pi(i-1)}+ \alpha_{i+1} e_{\pi(i+1)}.$$
As a consequence of the previous identities we have the relations: 
$$ \alpha_{1}^2=\alpha_{2}, \, \, \, \,  \textrm{and} \, \, \, \,  \alpha_{i-1} = \frac{\alpha_{i}^2}{2}  = \alpha_{i+1} \, \, \, \,  \textrm{for} \, \, \, \, i\in \{2, 3, \cdots, n-1\}.$$
In particular, for $i\in\{2,3\}$ we have 
$$ \alpha_{1} = \frac{\alpha_{2}^2}{2}  = \alpha_{3} \, \, \, \,  \textrm{and} \, \, \, \, \alpha_{2} = \frac{\alpha_{3}^2}{2}  = \alpha_{4} ,$$
then
$$\alpha_{2}= \frac{\alpha_{3}^2}{2} =\frac{ \left(\alpha_2^2 / 2\right)^2}{2} =\frac{\alpha_{2}^4}{8}.$$
Note that $\alpha_2 =0$ is a solution of the equation above. In this case we obtain $\alpha_i =0$, for any $i$, and therefore $g$ is the null homomorphism. If $\alpha_2 \neq 0$, it should be positive, and then we obtain $ \alpha_{2}  = 2  $. Finally, note that as $\alpha_1^2 = \alpha_2$ it should be $\alpha_1 = \sqrt{2}$, but $ \alpha_{1}  = \alpha_2^2 /2 =2$ and we get a contradiction. Therefore the only evolution homomorphism from $\mathcal{A}(A_n)$ to $\mathcal{A}_{RW}(A_n)$ is the null map and  $\A(A_n) \ncong \A_{RW}(A_n)$ as evolution algebras. Observe that our proof assumes $n>3$.
\end{proof}

\begin{remark}
It is not difficult to see that Proposition \ref{prop:path} is also true for a semi-infinite path with vertices $\{1,2,\ldots\}$, which we denote by $A_{\infty}$. The proof follows the same lines as before. More precisely, one can notice that the respective evolution algebras have set of generators given by $\{e_1,e_2,\ldots\}$ and relations:

\smallskip
$$
\mathcal{A}(A_{\infty}): \left\{
\begin{array}{cl}
 e_1^2= e_2,&\\[.2cm]
  e_i^2=e_{i-1}+e_{i+1}, & \text{ for } i\in \{2,\ldots\},\\[.2cm]
 e_i \cdot e_j =0, & \text{ for }i\neq j,
\end{array}\right.
$$
\smallskip
and

\smallskip
$$
\mathcal{A}_{RW}(A_{\infty}):\left\{
\begin{array}{cl}
 e_1^2= e_2,&\\[.2cm]
 e_i^2=\frac{1}{2}e_{i-1}+\frac{1}{2}e_{i+1}, & \text{ for } i\in \{2,\ldots\},\\[.2cm]
 e_i \cdot e_j =0,&  \text{ for }i\neq j. 
\end{array}\right.
$$
Then, it is not difficult to see that by assuming the existence of an evolution homomorphism $g:\mathcal{A}(A_{\infty})\longrightarrow \mathcal{A}_{RW}(A_{\infty})$, such that $g(e_i)=\sum_{k=1}^{\infty} t_{ik}e_k$, for any $i\in \{1,2,\ldots\}$, where the $t_{ik}$'s are scalars, one can obtain \eqref{eq:nova3}, \eqref{eq:nsqp6} and \eqref{eq:nsqp3} (for $k\in \{3,\ldots\}$). Then one can conclude from these equations $t_{i\ell}t_{j\ell}=0$ for $i\neq j$ and $\ell$ even. The same equality is obtained for $\ell$ odd by the same arguments of Case 1 (take $i_0=1$) and Case 3 (take $i_0\neq1$) in the proof of Proposition \ref{prop:path}. At this point, a contradiction is obtained when the equality $g(e_i^2)=g(e_i)\,g(e_i)$ is checked for $i\in \{1,2,3\}$.  This is an example of infinite graph such that  $\A(G) \ncong \A_{RW}(G)$.
\end{remark}

\begin{remark}
In Theorem \ref{teo:iso}(ii) we prove that the respective evolution algebras associated to any complete bipartite graph are isomorphic. A natural question arises when one remove edges in these graphs. In this case one can have the two different behaviors presented in this paper. By one hand, one can obtain an example of bipartite graph (non-complete) for which the respective algebras are isomorphic. Take for example a cycle graph, which is a $2$-regular graph, and apply Theorem \ref{teo:iso}(i). On the other hand it is possible to obtain an example of bipartite graph with non-isomorphic associated algebras. In this case take a path graph and apply Proposition \ref{prop:path}. See Figure \ref{FIG:bipartiteej} for more details.  
\end{remark}

\begin{figure}[h]
\label{FIG:bipartiteej}
\begin{center}

\subfigure[][Cycle graph $C_8$.]{

\begin{tikzpicture}[scale=0.8]

\draw (-2,3) -- (2,3)--(-2,2)--(2,2)--(-2,1)--(2,1)--(-2,0)--(2,0)--(-2,3);


\filldraw [black] (-2,1) circle (2.5pt);
\filldraw [black] (-2,0) circle (2.5pt);
\filldraw [black] (-2,2) circle (2.5pt);
\filldraw [black] (-2,3) circle (2.5pt);
\filldraw [black] (2,0) circle (2.5pt);
\filldraw [black] (2,1) circle (2.5pt);
\filldraw [black] (2,2) circle (2.5pt);
\filldraw [black] (2,3) circle (2.5pt);

\end{tikzpicture}

}\qquad\qquad
\subfigure[][Path Graph $A_7$.]{

\begin{tikzpicture}[scale=0.8]

\draw (-2,3) -- (2,2.5)--(-2,2)--(2,1.5)--(-2,1)--(2,0.5)--(-2,0);


\filldraw [black] (-2,1) circle (2.5pt);
\filldraw [black] (-2,0) circle (2.5pt);
\filldraw [black] (-2,2) circle (2.5pt);
\filldraw [black] (-2,3) circle (2.5pt);
\filldraw [black] (2,0.5) circle (2.5pt);
\filldraw [black] (2,1.5) circle (2.5pt);
\filldraw [black] (2,2.5) circle (2.5pt);

\end{tikzpicture}

}

\caption{Examples of bipartite graphs with different behaviors.}
\end{center}
\end{figure}
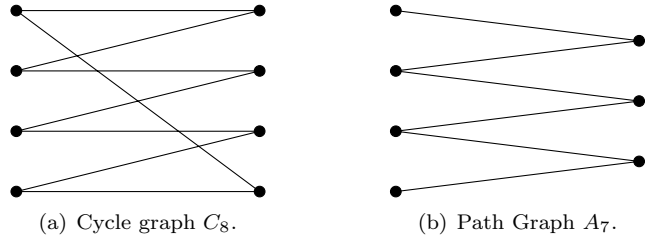

In the sequel we consider the friendship graph, usually denoted by $F_n$, which is a finite graph with $2n+1$ vertices, $3n$ edges, constructed by joining $n$ copies of the triangle graph with a common vertex (see Figure 3.3).

\begin{figure}[h]
\label{FIG:friendship}
\begin{center}

\subfigure[][$F_{2}$]{

\begin{tikzpicture}[scale=0.8]

\draw (0,0) -- (2,1)--(2,-1)--(0,0);
\draw (0,0) -- (-2,1)--(-2,-1)--(0,0);


\filldraw [black] (0,0) circle (2.5pt);
\filldraw [black] (2,1) circle (2.5pt);
\filldraw [black] (2,-1) circle (2.5pt);
\filldraw [black] (-2,-1) circle (2.5pt);
\filldraw [black] (-2,1) circle (2.5pt);
\filldraw [white] (-1,2) circle (2.5pt);
\filldraw [white] (-1,-2) circle (2.5pt);
\filldraw [white] (1,2) circle (2.5pt);
\filldraw [white] (1,-2) circle (2.5pt);

\end{tikzpicture}

}\qquad\qquad
\subfigure[][$F_{4}$]{

\begin{tikzpicture}[scale=0.8]

\draw (0,0) -- (2,1)--(2,-1)--(0,0);
\draw (0,0) -- (-2,1)--(-2,-1)--(0,0);
\draw (0,0) -- (-1,2)--(1,2)--(0,0);
\draw (0,0) -- (-1,-2)--(1,-2)--(0,0);


\filldraw [black] (0,0) circle (2.5pt);
\filldraw [black] (2,1) circle (2.5pt);
\filldraw [black] (2,-1) circle (2.5pt);
\filldraw [black] (-2,-1) circle (2.5pt);
\filldraw [black] (-2,1) circle (2.5pt);
\filldraw [black] (-1,2) circle (2.5pt);
\filldraw [black] (-1,-2) circle (2.5pt);
\filldraw [black] (1,2) circle (2.5pt);
\filldraw [black] (1,-2) circle (2.5pt);

\end{tikzpicture}

}\qquad\qquad
\subfigure[][$F_{8}$]{

\begin{tikzpicture}[scale=0.8]

\draw (0,0) -- (2,0.8)--(2,-0.8)--(0,0);
\draw (0,0) -- (-2,0.8)--(-2,-0.8)--(0,0);
\draw (0,0) -- (-0.8,2)--(0.8,2)--(0,0);
\draw (0,0) -- (-0.8,-2)--(0.8,-2)--(0,0);
\draw (0,0) -- (-1.8,1.1)--(-1.1,1.8)--(0,0);
\draw (0,0) -- (1.8,1.1)--(1.1,1.8)--(0,0);
\draw (0,0) -- (-1.8,-1.1)--(-1.1,-1.8)--(0,0);
\draw (0,0) -- (1.8,-1.1)--(1.1,-1.8)--(0,0);

\filldraw [black] (0,0) circle (2.5pt);
\filldraw [black] (2,0.8) circle (2.5pt);
\filldraw [black] (2,-0.8) circle (2.5pt);
\filldraw [black] (-2,-0.8) circle (2.5pt);
\filldraw [black] (-2,0.8) circle (2.5pt);
\filldraw [black] (-0.8,2) circle (2.5pt);
\filldraw [black] (-0.8,-2) circle (2.5pt);
\filldraw [black] (0.8,2) circle (2.5pt);
\filldraw [black] (0.8,-2) circle (2.5pt);

\filldraw [black] (1.1,-1.8) circle (2.5pt);
\filldraw [black] (1.1,1.8) circle (2.5pt);
\filldraw [black] (-1.1,-1.8) circle (2.5pt);
\filldraw [black] (-1.1,1.8) circle (2.5pt);

\filldraw [black] (-1.8,1.1) circle (2.5pt);
\filldraw [black] (1.8,1.1) circle (2.5pt);
\filldraw [black] (-1.8,-1.1) circle (2.5pt);
\filldraw [black] (1.8,-1.1) circle (2.5pt);
\end{tikzpicture}
}

\caption{Friendship graph.}
\end{center}
\end{figure}
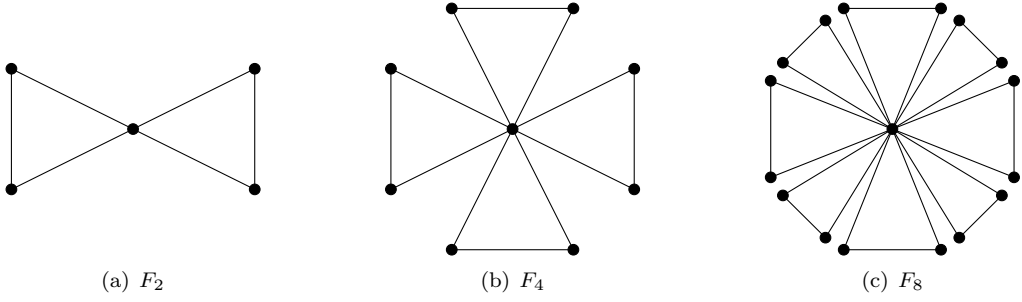

We will slightly abuse the notation, and write $\tilde{F}_n$ to the friendship graph with $n$ vertices, i.e. $\tilde{F}_n := F_{(n-1)/2}$, for $n\in \{2k+1,k\geq 1\}$. By Theorem \ref{teo:iso}(i) we have $\A(\tilde{F}_3)\cong \A_{RW}(\tilde{F}_3)$. Here we consider $n>4$.

 \begin{prop}
Let $\tilde{F}_{n}$ be a friendship graph, with $n>4$. Then the only evolution homomorphism between $\A(\tilde{F}_{n})$ and $\A_{RW}(\tilde{F}_{n})$ is the null map. Thus $\A(\tilde{F}_{ n}) \ncong \A_{RW}(\tilde{F}_{n})$ as evolution algebras.
\end{prop}

\begin{proof}
Suppose that $\A(\tilde{F}_{ n})$ and  $ \A_{RW}(\tilde{F}_{n})$ are the evolution algebras induced by $\tilde{F}_{n}$ and by the random walk on $\tilde{F}_{n}$, respectively. That is, take the set of generators $\{e_1,e_2,\ldots,e_n\}$, and the relations:

\smallskip
$$
\mathcal{A}(\tilde{F}_{n}):\left\{
\begin{array}{cl}
 e_i^2=e_{i+1}+e_{n}, &  \textrm{for} \, \,  i\in\{1,\ldots,n-1\} \, \,  \textrm{such that }i \text{ is odd,}\\[.5cm]
 e_i^2=e_{i-1}+e_{n}, &  \textrm{for} \, \,  i\in\{1,\ldots,n-1\} \, \,  \textrm{such that }i \text{ is even,}\\[.5cm]
 e_n^2= \displaystyle\sum_{i=1} ^{n-1} e_{i}, &  \\[.5cm]
 e_i \cdot e_j =0, & \text{ for } i\neq j, 
\end{array}\right.
$$
and
\smallskip
$$
\mathcal{A}_{RW}(\tilde{F}_{n}):\left\{
\begin{array}{cl}
 e_i^2=\displaystyle\frac{1}{2}(e_{i+1}+e_{n}), &  \textrm{for} \, \,  i\in\{1,\ldots,n-1\} \, \,  \textrm{such that }i \text{ is odd,}\\[.5cm]
 e_i^2=\displaystyle\frac{1}{2}(e_{i-1}+e_{n}), &  \textrm{for} \, \,  i\in\{1,\ldots,n-1\} \, \,  \textrm{such that }i \text{ is even,}\\[.5cm]
 e_n^2= \displaystyle \frac{1}{n-1}\displaystyle \sum_{i=1} ^{n-1} e_{i}, &  \\[.5cm]
 e_i \cdot e_j =0, & \text{ for } i\neq j. 
\end{array}\right.
$$
\smallskip

 Let $g:\mathcal{A}(\tilde{F}_{n})\longrightarrow \mathcal{A}_{RW}(\tilde{F}_{n})$ be an evolution homomorphism such that
\begin{equation} \nonumber
g(e_i)=\sum_{k=1}^{n} t_{ik}e_k, \text{ for } i\in\{1,2,\ldots,n\},
\end{equation}
\noindent
where the $t_{ik}$'s are scalars. Thus $g(e_i)\cdot g(e_j)=0$ for $i\neq j$. But
\begin{eqnarray}
\nonumber g(e_i)\cdot g(e_j)&=& \sum_{k=1}^{n}t_{ik}t_{jk} e^2_{k}\\
\nonumber        &=& \left(\frac{t_{i2}t_{j2}}{2}+\frac{t_{in}t_{jn}}{n-1} \right)e_1 +\left(  \frac{t_{i1}t_{j1}}{2}+\frac{t_{in}t_{jn}}{n-1} \right)e_2 + \cdots  + 
	\left(\frac{1}{n-1} \sum_{k=1}^{n-1} t_{ik}t_{jk}\right)e_{n}.
\end{eqnarray}
\noindent
Then the coefficients of $e_{\ell}$ are zero for $\ell\in\{1, \ldots, n\}$.  According to the above remark, we have 

\begin{eqnarray}
\label{Mwhe:1}\frac{t_{i(\ell+1)}t_{j(\ell+1)}}{2}+ \frac{t_{in}t_{jn}}{n-1} &=& 0,\,\,\,\,\, \,\,\,\,\,\,\textrm{for} \, \,  \ell\in\{1,\ldots,n-1\} \, \,  \textrm{such that }\ell \text{ is odd,}  \\[.2cm]
\label{Mwhe:2}\frac{t_{i(\ell-1)}t_{j(\ell-1)}}{2}+ \frac{t_{in}t_{jn}}{n-1} &=& 0,\,\,\,\,\, \,\,\,\,\,\,\textrm{for} \, \,  \ell\in\{1,\ldots,n-1\} \, \,  \textrm{such that }\ell \text{ is even,}\,\,\,\,\,\,\,\,\\[.2cm]
\label{Mwhe:4} \frac{1}{n-1}\sum_{k=1}^{n-1} t_{i k}t_{j k}&=&0, \,\,\,\,\, \,\,\,\,\,\,\text{ for }  \ell=n. \,\,\,\,\,\,\,\,
\end{eqnarray}

\noindent
Adding the equations (\ref{Mwhe:1}) and (\ref{Mwhe:2}), for $\ell \in \{1,\ldots,n-1\}$, we get 

$$ \left(\ \frac{1}{2} \sum_{\ell=1}^{n-1} t_{i\ell}t_{j\ell}\right)+(n-1) \left(\frac{t_{in}t_{jn}}{n-1}\right) =0,$$ 
which, using  (\ref{Mwhe:4}), implies
\begin{equation*} t_{in}t_{jn}=0, \text { for } i,j\in \{1,\ldots, n\}\,\,\,\,  \text{and} \,\,\,\,i\not= j.
\end{equation*}
This in turn implies, by (\ref{Mwhe:1}) and   (\ref{Mwhe:2}), that 
\begin{equation*} t_{ik}t_{jk}=0, \text { for } i,j,k\in \{1,\ldots, n\}\,\,\,\,  \text{and} \,\,\,\,i\not= j.
\end{equation*}

Then we obtain that $t_{ik}\neq 0$ for at most one of the values of $i$, and for any $k$. These  means that if $g$ exists,  it should be defined as
$$g(e_i) = \alpha_i e_{\pi(i)},$$
for $i\in\{1,2,\ldots,n\}$, where the $\alpha_i's$ are scalars and $\pi \in S_n$.

Again, by our assumption we have  $$g(e_i^2)=g(e_i)\cdot g(e_i)= \alpha_i^2 e_{\pi(i)}^2, \,  \, \, \textrm{for } \, \, i  \in \{1, \ldots, n\},$$ and
$$g(e_i^2) = g(e_{i+t}+e_{n})=\alpha_{i+t} e_{\pi(i+t)}+ \alpha_{n} e_{\pi(n)}, \, \, \textrm{for} \, \,  i \neq n,$$ 
where 
\begin{equation}\label{eq:t}
t:=t(i)=\left\{
\begin{array}{cl}
+1,&\text{ if }i\text{ is odd},\\[.2cm]
-1,&\text{ if }i\text{ is even.}
\end{array}\right.
\end{equation}
Thus
$$\alpha_i^2 e_{\pi(i)}^2 = \alpha_{i+t} e_{\pi(i+t)}+ \alpha_{n} e_{\pi(n)}, \, \, \textrm{for} \, \, i \neq n. $$

If $\pi(n) \neq n$, then $\pi(i+t)= n$ and 
$$\alpha_i^2 e_{\pi(i)}^2 =  \alpha_{i+t} e_{n}+ \alpha_{n} e_{\pi(n)}, $$
with $\pi(i), \pi(n) \in \{1,2, \cdots, n-1\}$.
As $e_{\pi(i+t)}=e_{n}$ we obtain 
$$ g(e_{\pi(i+t)}^2) = g(e_{n}^2) = g \left( \frac{1}{n-1} \sum_{i=1}^{n-1} e_{i} \right) =   \frac{1}{n-1} \left(\sum_{i=1}^{n-1} \alpha_{i}  e_{\pi(i)} \right), $$
and also 
$$  g(e_{\pi(i+t)}^2) = g(e_{\pi(i+t)}) \cdot g(e_{\pi(i+t)}) = \alpha_{n}^2 e_{\pi(n)}^2 = \frac{\alpha_{n}^2}{2} \left(  e_{\pi(n)+t} + e_{n} \right).$$
This gives $\alpha_{n}^2=0$, thus $\alpha_{n}=0$ and then $\alpha_{i}=0$ for $i \in \{1, \ldots, n-1\}$.  We conclude that $g$ is a null homomorphism. 

On the other hand, if $\pi(n) = n$ then,  for $i \neq n$, we have
$$g(e_i^2) = g(e_{i+t}+e_{n})=\alpha_{i+t} e_{\pi(i+t)}+ \alpha_{n} e_{n},$$
where $t:=t(i)$ is defined by \eqref{eq:t} and 
$$\alpha_i^2 e_{\pi(i)}^2 = \alpha_{i+t} e_{\pi(i+t)}+ \alpha_{n} e_{n}. $$
But also
$$\alpha_i^2 e_{\pi(i)}^2 = \frac{\alpha_{i}^2}{2}( e_{\pi(i)+t}+ \alpha_{n} e_{n}), \, \, \,  \textrm{with} \, \, \, \pi(i), \pi(i)+t \in \{1,2, \cdots, n-1\}. $$
Then 
$$\frac{\alpha_i^2}{2} \left( e_{\pi(i)+t} + e_{n} \right) = \alpha_{i+t} e_{\pi(i)+t}+ \alpha_{n} e_{n},$$
$\alpha_{n}=\alpha_{i+t}=\alpha_i^2/2$ for $i\in\{1, \ldots, n-1\}$. Therefore $\alpha_{i}=\alpha_{j}$ for $i \neq j$ and $i,j \in \{1, \cdots, n\}$. It follows that $\alpha_i^2/2=\alpha_{i}$. Note that if $\alpha_i =0$ then $g=0$. Lets assume $\alpha_i\neq 0$. Hence $\alpha_{i}=2$ for $i \in \{1, \ldots, n\}$,
$$g(e_n^2) = \alpha_{n}^2 e_{n}^2 = 4 e_{n}^2 = \frac{4}{n-1} \left(\sum_{i=1}^{n-1}  e_{i} \right),$$
and 
$$g(e_n^2) = g \left(\sum_{i=1}^{n-1} e_{i} \right) = 2 \sum_{i=1}^{n-1} e_{\pi(i)}. $$
Therefore $4/(n-1)=2$ and this is possible only if $n=3$. By hypothesis $n>4$, and
this implies that the only evolution homomorphism from $\mathcal{A}(\tilde{F}_n)$ to $\mathcal{A}_{RW}(\tilde{F}_n)$ is the null map. Therefore $\A(\tilde{F}_n) \ncong \A_{RW}(\tilde{F}_n)$ as evolution algebras. 

\end{proof}


Now we consider the wheel graph $W_n$, which is a graph with $n$ vertices, $n\geq 4$, formed by connecting a single vertex, called center, to all the vertices of an $(n-1)$-cycle (see Figure 3.4). Since $W_4$ is a $3$-regular graph we know that $\A(W_4) \cong\A_{RW}(W_4)$ (see Theorem \ref{teo:iso}(i)).

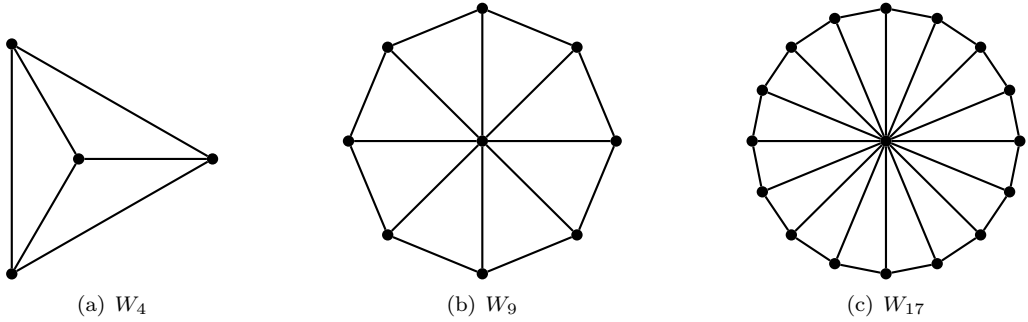
\begin{figure}[h]
\begin{center}

\subfigure[][$W_{4}$]{

\begin{tikzpicture}[scale=.8]
\GraphInit[vstyle=Simple]
    \SetGraphUnit{1}
    \tikzset{VertexStyle/.style = {shape = circle,fill = black,minimum size = 2.5pt,inner sep=1.5pt}}
  \begin{scope}[xshift=12cm]
\grWheel[RA=2.2]{4}%
\end{scope}
\end{tikzpicture}

}\qquad\qquad
\subfigure[][$W_{9}$]{

\begin{tikzpicture}[scale=.8]
\GraphInit[vstyle=Simple]
    \SetGraphUnit{1}
    \tikzset{VertexStyle/.style = {shape = circle,fill = black,minimum size = 2.5pt,inner sep=1.5pt}}
  \begin{scope}[xshift=12cm]
\grWheel[RA=2.2]{9}%
\end{scope}
\end{tikzpicture}

}\qquad\qquad
\subfigure[][$W_{17}$]{

\begin{tikzpicture}[scale=.8]
\GraphInit[vstyle=Simple]
    \SetGraphUnit{1}
    \tikzset{VertexStyle/.style = {shape = circle,fill = black,minimum size = 2.5pt,inner sep=1.5pt}}
  \begin{scope}[xshift=12cm]
\grWheel[RA=2.2]{17}%
\end{scope}
\end{tikzpicture}

}
\caption{Wheel graphs.}
\end{center}
\label{FIG:wheel}
\end{figure}

 \begin{prop}
Let $W_{n}$ be a wheel graph, with $n>4$. Then the only evolution homomorphism between $\A(W_n)$ and $\A_{RW}(W_{n})$ is the null map. In particular, $\A(W_{n}) \ncong \A_{RW}(W_{n})$ as evolution algebras.
\end{prop}

\begin{proof}
In order to define the evolution algebras induced by $W_n$ and by the random walk on $W_n$, respectively, we consider the set of generators $\{e_1, e_2, \ldots, e_n\}$, and the relations given by 

\smallskip

$$
\mathcal{A}(W_n): \left\{
\begin{array}{cl}
 e_1^2= e_{n-1}+ e_{2} +e_{n} , & \\[.2cm]
 e_i^2= e_{i-1}+ e_{i+1}+e_{n}, & \text{ for } i\in \{2,\ldots, n-2\},\\[.2cm]
 e_{n-1}^2= e_{1}+e_{n-2} +e_{n} , & \\[.2cm]
 e_n^2=\displaystyle \sum_{j=1}^{n-1} e_{j}, & \\[.2cm]
 e_i \cdot e_j =0, & \text{ for } i\neq j, 
\end{array}\right.
$$
and
\smallskip
$$
\mathcal{A}_{RW}(W_n):\left\{
\begin{array}{cl}
 e_1^2= \frac{1}{3}(e_{n-1}+ e_{2}+e_{n}) , & \\[.2cm]
 e_i^2= \frac{1}{3}(e_{i-1}+ e_{i+1}+e_{n}), &  \text{ for } i\in \{2,\ldots, n-2\},\\[.2cm]
 e_{n-1}^2= \frac{1}{3}(e_{1}+e_{n-2} +e_{n}) , & \\[.5cm]
 e_n^2= \frac{1}{n-1}\displaystyle \left ( \sum_{j=1}^{n-1} e_{j} \right), & \\[.2cm]
 e_i \cdot e_j =0, & \text{ for } i\neq j. 
\end{array}\right.
$$
Assume that there exists an evolution homomorphism $g:\mathcal{A}(W_n)\longrightarrow \mathcal{A}_{RW}(W_n)$, such that for $i\in  \{1,2,\ldots,n\}$
\begin{equation} 
\nonumber g(e_i)=\sum_{k=1}^{n} t_{ik}e_k,
\end{equation}
where the $t_{ik}'s$ are scalars. Thus defined, we have 

\begin{eqnarray}
\nonumber g(e_i)\cdot g(e_j)&=& \sum_{k=1}^{n}t_{ik}t_{jk} e^2_{k}\\
\nonumber        &=& \left(\frac{t_{i2}t_{j2}}{3}+\frac{t_{i(n-1)}t_{j(n-1)}}{3}+ \frac{t_{in}t_{jn}}{n-1} \right)e_1\\
\nonumber								&& +\sum_{\ell=2}^{n-2}\left( \frac{t_{i(\ell-1)}t_{j(\ell-1)}}{3}+  \frac{t_{i(\ell+1)}t_{j(\ell+1)}}{3}+ 
									\frac{t_{in}t_{jn}}{n-1}\right)e_{\ell}\\
\label{whe:0}		&& + \left(\frac{t_{i1}t_{j1}}{3}+\frac{t_{i(n-2)}t_{j(n-2)}}{3}+ \frac{t_{in}t_{jn}}{n-1} \right)e_{n-1}+
									\left(\sum_{\ell=1}^{n-1} \frac{t_{i\ell}t_{j\ell}}{3}\right)e_{n},
\end{eqnarray}
which implies, whether $i\neq j$ because $g(e_i)\cdot g(e_j)=0$, that

\begin{eqnarray}
\label{whe:1}\frac{t_{i2}t_{j2}}{3}+\frac{t_{i(n-1)}t_{j(n-1)}}{3}+ \frac{t_{in}t_{jn}}{n-1}&=& 0,  \\[.2cm]
\label{whe:2}\frac{t_{i(\ell-1)}t_{j(\ell-1)}}{3}+ \frac{t_{i(\ell+1)}t_{j(\ell+1)}}{3}+ \frac{t_{in}t_{jn}}{n-1} &=& 0,\,\,\,\,\, \,\,\,\,\,\,\text{ for }  \ell\in \{2,\ldots,n-2\}, \,\,\,\,\,\,\,\,\\[.2cm]
\label{whe:3}\frac{t_{i1}t_{j1}}{3}+\frac{t_{i(n-2)}t_{j(n-2)}}{3}+ \frac{t_{in}t_{jn}}{n-1}&=& 0, \\[.2cm]
\label{whe:4}\sum_{\ell=1}^{n-1} t_{i\ell}t_{j\ell}&=&0. 
\end{eqnarray}

Adding the equalities (\ref{whe:1}), (\ref{whe:2}) for $\ell\in \{2,\ldots,n-2\}$, and (\ref{whe:3}) we obtain

$$ t_{in}t_{jn}+2 \sum_{\ell=1}^{n-1}\frac{ t_{i\ell}t_{j\ell}} {3}=0.$$ 

So  we conclude by (\ref{whe:4}) that
\begin{equation} \label{eq:tinwheel}
t_{in}t_{jn}=0, \text { for } i,j\in \{1,\ldots, n\}\,\,\,\,  \text{and} \,\,\,\,i\not= j.
\end{equation}

On the other hand,
\begin{equation}\label{whe:6}
g(e_i^2)= \left\{\begin{array}{rlll}
g(e_{\ell_1(i)}+e_{\ell_2(i)}+e_{n}) &=&\displaystyle \sum_{k=1}^{n}(t_{\ell_{1(i)}k}+t_{\ell_{2(i)}k}+t_{nk})e_k,  &\text{ for }  i \not= n ,\\[.2cm]
g\left(\displaystyle \sum_{j=1}^{n-1}e_j\right)&=&\displaystyle \sum_{k=1}^{n}(t_{1k} +t_{2k}+\cdots + t_{(n-1)k} )e_{k}, &\text{ for }  i = n, 
\end{array}\right.
\end{equation}

\smallskip
\noindent where $\ell_{1(i)}$ and $\ell_{2(i)}$ are the neighbors of the vertex $i$, i.e.,
$$
\ell_{1(i)}:=\left\{
\begin{array}{cl}
i-1,& \text{for }i\in \{2,\ldots,n-1\},\\[.2cm]
n-1,&\text{for }i=1,\\[.2cm]
\end{array}\right.
$$
and
$$
\ell_{2(i)}:=\left\{
\begin{array}{cl}
i+1,& \text{for }i\in \{1,\ldots,n-2\},\\[.2cm]
1,&\text{for }i=n-1.\\[.2cm]
\end{array}\right.
$$
Therefore using  (\ref{whe:0}) we know that the $n$th-coordinate of $g(e_i^2)$ in the natural basis $\{e_1, \ldots, e_n\}$
is 

\smallskip
\begin{equation}
 \nonumber \frac{t_{i1}^2}{3} + \frac{t_{i2}^2}{3}+\cdots +\frac{t_{i(n-1)}^2}{3},\,\,\,\text{for }i\in\{1,2,\ldots,n\}.
\end{equation}
Then 

\begin{eqnarray}
\label{whe:7} \frac{t_{i1}^2}{3} + \frac{t_{i2}^2}{3}+\cdots +\frac{t_{i(n-1)}^2}{3} &=& t_{\ell_{1(i)}n}+t_{\ell_{2(i)}n}+t_{nn},   \text{ for } i\in \{1, \ldots, n-1 \}, \\
\label{whe:8} \frac{t_{n1}^2}{3} + \frac{t_{n2}^2}{3}+\cdots +\frac{t_{n(n-1)}^2}{3} &=& t_{1n}+t_{2n}+ \cdots +t_{(n-1)n}.
\end{eqnarray}

\smallskip
In what follow we shall assume $t_{nn}\neq 0$. In such case, by (\ref{eq:tinwheel}) we have that 
\begin{equation}\label{whe:00}
t_{in}=0 \text{ for } i \in\{1, \ldots,n-1\}. 
\end{equation}
This, together with Equation (\ref{whe:8}) implies

\begin{equation}\label{tnj=0}
t_{n1}=t_{n2} = \cdots =t_{n(n-1)} =0.
\end{equation}

Adding the equalities of (\ref{whe:7}),  for $i\in \{1, 2, \ldots, n-1\}$, and using (\ref{whe:00}) we obtain  that
\begin{equation} \label{whe:n9}
 \frac{1}{3}\sum_{i,j=1}^{n-1}t_{ij}^2= (n-1)t_{nn}.
\end{equation}

Substituting (\ref{whe:00}) in  (\ref{whe:0}) and (\ref{whe:6}) and adding the $n-1$ first coordinates of $g(e_i ^2)$, for $i \in \{ 1, \ldots, n-1\}$, we  obtain the following $n-1$ equalities

\begin{eqnarray}
\nonumber 2\sum_{\ell=1} ^{n-1} \frac{ t_{1\ell}^2}{3} &=&  \sum_{k=1} ^{n}(t_{2k}+t_{(n-1)k}),\\
\nonumber 2\sum_{\ell=1} ^{n-1} \frac{ t_{2\ell}^2}{3} &=&  \sum_{k=1} ^{n}(t_{3k}+t_{1k}),\\
\nonumber  &   \vdots &  \\
\nonumber 2\sum_{\ell=1} ^{n-1} \frac{ t_{(n-1)\ell}^2}{3} &=&  \sum_{k=1} ^{n}(t_{1k}+t_{(n-2)k}).
\end{eqnarray}

Adding the previous equalities we have 

\begin{equation} \nonumber
\frac{2}{3} \sum_{\ell =1}^{n-1} \left(t_{1 \ell}^2 + t_{2 \ell}^2 + \cdots + t_{(n-1) \ell}^2\right)=2\sum_{k=1}^{n}(t_{1k}+t_{2k}+\cdots +t_{(n-1)k}).
\end{equation}

By repeating the above procedure for  $i=n$ we obtain
\begin{equation} \nonumber
t^{2}_{nn}=\sum_{i,j=1}^{n-1}t_{ij}.
\end{equation}
Therefore we get by \eqref{whe:n9}
\begin{equation} \nonumber
3(n-1)t_{nn}=\sum_{\ell =1}^{n-1} \left(t_{1 \ell}^2 + t_{2 \ell}^2 + \cdots + t_{(n-1) \ell}^2\right)=3\sum_{k=1}^{n}\left(t_{1k}+t_{2k}+\cdots +t_{(n-1)k}\right)=3t^2_{nn},
\end{equation}
which implies $t_{nn}>0$, and as a consequence 
\begin{equation} \label{whe:09}
t_{nn}=(n-1).
\end{equation}

Observe that by replacing \eqref{whe:09} in \eqref{whe:n9} we obtain 

\begin{equation} \label{whe:n1}
\sum_{i,j=1}^{n-1}t_{ij}^2= 3(n-1)^2,
\end{equation}
\noindent
which will be important to get a contradiction. Now, we use \eqref{whe:0} (for $i=j=n$) which, together with \eqref{tnj=0} and \eqref{whe:09}, to get
$$g(e_n)\cdot g(e_n) = (n-1)(e_1 + \cdots + e_{n-1}).$$

On the other hand, since $g(e_n)\cdot g(e_n) = g(e_n^2)$ we have by \eqref{whe:6} (for $i=n$) combined with the previous expression that for $k\in \{1,2,\ldots, n\}$ it should be

\begin{equation*}\label{whe:n2}
t_{1k} + \cdots +t_{(n-1) k} = n-1.
\end{equation*}

From now on we will work with some expressions resulting from

\begin{equation}\label{whe:n3}
(t_{1k} + \cdots +t_{(n-1) k} )^2= (n-1)^2.
\end{equation}
 
For the sake of clarity we separate the analysis into three steps:

\begin{enumerate}
\item[{\bf i.}] Add the equalities resulting by letting $k=2$ and $k=n-1$ in \eqref{whe:n3}. Then

\begin{equation}\label{whe:n4}
\sum_{\ell =1}^{n-1} t_{\ell 2}^2 + \sum_{\ell =1}^{n-1} t_{\ell (n-1)}^2 + 2 \underbrace{ \sum_{i < \ell}^{n-1} (t_{i 2}t_{\ell 2} + t_{i (n-1)} t_{\ell (n-1)})}_{=0 \text{ by }\eqref{whe:1} \text{ and }\eqref{whe:00}} = 2 (n-1)^2.
\end{equation} 

\item[{\bf ii.}] Take $j\in \{2,\ldots,n-2\}$ and add the equalities resulting by letting $k=j-1$ and $k=j+1$ in \eqref{whe:n3}. Then, for $j\in \{2,\ldots,n-2\}$, it holds

\begin{equation}\label{whe:n5}
\sum_{\ell =1}^{n-1} t_{\ell (j-1)}^2 + \sum_{\ell =1}^{n-1} t_{\ell (j+1)}^2 + 2 \underbrace{ \sum_{i < \ell}^{n-1} (t_{i (j-1)}t_{\ell (j-1)} + t_{i (j+1)} t_{\ell (j+1)})}_{=0 \text{ by }\eqref{whe:2} \text{ and }\eqref{whe:00}} = 2 (n-1)^2.
\end{equation} 

\item[{\bf iii.}]  Add the equalities resulting by letting $k=1$ and $k=n-2$ in \eqref{whe:n3}. Then

\begin{equation}\label{whe:n6}
\sum_{\ell =1}^{n-1} t_{\ell 1}^2 + \sum_{\ell =1}^{n-1} t_{\ell (n-2)}^2 + 2 \underbrace{ \sum_{i < \ell}^{n-1} (t_{i 1}t_{\ell 1} + t_{i (n-2)} t_{\ell (n-2)})}_{=0 \text{ by }\eqref{whe:3} \text{ and }\eqref{whe:00}} = 2 (n-1)^2.
\end{equation} 
\end{enumerate}

Having {\bf i - iii} we add the equalities \eqref{whe:n4}, \eqref{whe:n5} (for $j\in \{2,\ldots,n-2\}$) and \eqref{whe:n6}. Therefore, we obtain
$$2 \underbrace{\left(\sum_{\ell =1}^{n-1} t_{\ell 1}^2 +\cdots + \sum_{\ell =1}^{n-1} t_{\ell (n-1)}^2 \right)}_{=3(n-1)^2 \text{ by }\eqref{whe:n1}} = 2(n-1)^3,$$
i.e., we obtain $3(n-1)^2=(n-1)^3$ which is a contradiction because it only holds for $n=4$. Therefore, it should be $t_{nn}=0$.

Notice that if $t_{nn}=0$ and $t_{in}=0$ for all $i \in \{1,\ldots, n-1\}$ then by (\ref{whe:7})  and (\ref{whe:8})  we have $g=0$. In the other case, if $t_{nn}=0$ and there exists $i_{0}\in\{1,\ldots,n-1\}$ such that  $t_{i_{0}n}\not =0$ we obtain, by (\ref{whe:7}), $t_{i_{0}1}=t_{i_{0}2}=\cdots=t_{i_{0}(n-1)}=0$ so $g(e_{i_{0}})=t_{i_{0}n}e_{i_{0}}$.  Then, $$g(e_{i_{0}}^2)= t_{i_{0}n}^2(e_{\ell_{1(i_{0})}}+ e_{\ell_{2(i_{0})}}+e_n ).$$ On the  other hand,  by (\ref{whe:6}),
$$g(e_{i_{0}}^2)=\sum_{k=1}^n (t_{\ell_{1(i_{0})}k}+ t_{\ell_{(i_{0})}k}+t_{nk})e_k,$$
so $t^{2}_{i_{0}n}=t_{\ell_{1(i_{0})}n}+ t_{\ell_{(i_{0})}n}+t_{nn}=0$, and therefore $t_{i_{0}n}=0$.

We conclude  that the only evolution homomorphism between $\A(W_n)$ and $\A_{RW}(W_n)$, for $n>4$, is the null map. In particular $\A(W_n) \not \cong \A_{RW}(W_n)$.


\end{proof}
 

\subsection{Complete $n$-partite graphs}

A natural generalization of the complete bipartite graph is the complete $n$-partite graph, for $n\geq 2$, with partitions of sizes $a_1, a_2, \ldots, a_n$, where $a_i \geq 1$ for $i\in \{1,2,\ldots,n\}$. This graph, which we denote by $K_{a_1, a_2, \ldots, a_n}$, has a set of vertices partitioned into $n$ disjoint sets of sizes $a_1, a_2, \ldots, a_n$,  respectively, in such a way that there is no edge connecting two vertices in the same subset, and every possible edge that could connect vertices in different subsets is part of the graph. 

\smallskip
The resulting evolution algebra associated to the graph $\mathcal{A}(K_{a_1, a_2, \ldots, a_n})$ is given by the generator set $\{e_1,\ldots,e_{a_1},e_{a_1 +1},\ldots ,e_{a_1 + a_2}, \ldots, e_{a_1 + \cdots + a_n}  \}$ and the relations:
\smallskip

\[
\begin{array}{ll}
 e_i^2   = \displaystyle \sum_{j=1}^{ a_2+\cdots + a_n} e_{a_1 + j}, & \text{for }i\in \{1,\ldots,a_1\},\\[.6cm]
 e_{a_1+\cdots +a_{t-1}+i}^2    =\displaystyle  \sum_{j=1}^{a_1+\cdots + a_{t-1}} e_{j} + \sum_{j=1}^{a_{t+1}+\cdots + a_{n}} e_{a_1+\cdots +a_{t}+j} ,&\begin{array}{l}
 \text{for }t\in \{2,\ldots,n-1\}\\
\text{and } i\in \{1,\ldots,a_t\},
 \end{array}\\[.6cm]
  e_{a_1+\cdots +a_{n-1}+i}^2   =\displaystyle \sum_{j=1}^{ a_1+\cdots + a_{n-1}} e_{ j}, &  \text{for }i\in \{1,\ldots,a_n\},\\[.6cm]
  e_i \cdot e_j =  0 , &\text{for }i\neq j.
\end{array}
\]
On the other hand, if we let $s=\sum_{k=1}^{n} a_{k}$, the evolution algebra associated to the random walk on the graph, denoted by $\mathcal{A}_{RW}(K_{a_1, a_2, \ldots, a_n})$, is given by the same set of generators as before and the relations:

\[
\begin{array}{lll}
 e_i^2   = \displaystyle \frac{1}{s-a_1} \left(\sum_{j=1}^{ a_2+\cdots + a_n} e_{a_1 + j}\right), & \text{for }i\in \{1,\ldots,a_1\},\\[.6cm]
 e_{a_1+\cdots +a_{t-1}+i}^2    =\displaystyle  \frac{1}{s-a_t}\left(\sum_{j=1}^{a_1+\cdots + a_{t-1}} e_{j} + \sum_{j=1}^{a_{t+1}+\cdots + a_{n}} e_{a_1+\cdots +a_{t}+j}\right) ,& 
 \begin{array}{l}
 \text{for }t\in \{2,\ldots,n-1\}\\
\text{and } i\in \{1,\ldots,a_t\},
 \end{array}\\[.6cm]
  e_{a_1+\cdots +a_{n-1}+i}^2   =\displaystyle \frac{1}{s-a_n}\left(\sum_{j=1}^{ a_1+\cdots + a_{n-1}} e_{ j}\right), &  \text{for }i\in\{1,\ldots,a_n\},\\[.6cm]
  e_i \cdot e_j =  0 , &\text{for }i\neq j.
\end{array}
\]

\smallskip
As a consequence of Theorem \ref{teo:iso} (i) we have that $\mathcal{A}(K_{a_1, a_2, \ldots, a_n})\cong \mathcal{A}_{RW}(K_{a_1, a_2, \ldots, a_n})$ as evolution algebras, provided $a_i =d$ for any $i$, where $d\geq 2$ is a given constant. On the other hand, we arrive at a similar conclusion for $n=2$ and any value of $a_i$ by Theorem \ref{teo:iso} (ii). As we show in the next example, this is not true in general.

\begin{exa}
Let $K_{1,1,2}$ be the complete $3$-partite graph, with partitions of sizes $1,1$ and $2$, see Figure 3.5.

\begin{figure}[h]
\label{FIG:3partite}
\begin{center}
\begin{tikzpicture}[scale=0.8]

\draw (-1.5,1.5) -- (1.5,1.5)--(1,-1.5)--(-1.5,1.5)--(-1,-1.5)--(1.5,1.5);


\filldraw [black] (-1.5,1.5) circle (2.5pt);
\draw (-1.5,1.7) node[above,font=\footnotesize] {$1$};
\filldraw [black] (1.5,1.5) circle (2.5pt);
\draw (1.5,1.7) node[above,font=\footnotesize] {$2$};
\filldraw [black] (-1,-1.5) circle (2.5pt);
\draw (-1,-1.7) node[below,font=\footnotesize] {$3$};
\filldraw [black] (1,-1.5) circle (2.5pt);
\draw (1,-1.7) node[below,font=\footnotesize] {$4$};

\end{tikzpicture}
\end{center}
\caption{Complete $3$-partite graph $K_{1,1,2}$.}
\end{figure}
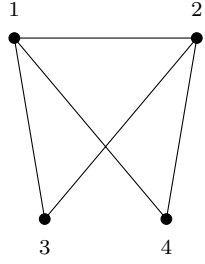
The resulting evolution algebras associated to $K_{1,1,2}$ are given by the set of generators $\{e_1, e_2,e_3,e_4\}$ and relations 
$$
\mathcal{A}(K_{1,1,2}):\left\{
\begin{array}{cl}
 e_1^2= e_2 + e_3+e_4,\\[.5cm]
 e_2^2=  e_1 + e_3+e_4,  \\[.5cm]
 e_i^2 = e_1+e_2 , & \text{ for }i\in \{3,4\}, \\[.5cm]
 e_i \cdot e_j =0, & \text{ for } i\neq j. 
								\end{array}\right.
$$
and 
\smallskip
$$
\mathcal{A}_{RW}(K_{1,1,2}):\left\{
\begin{array}{cl}
                      e_1^2=\displaystyle \frac{1}{3}\left( e_2 + e_3+e_4\right),\\[.5cm]
 e_2^2=\displaystyle \frac{1}{3}\left(e_1 + e_3+e_4\right),\\[.5cm]
 										e_i^2 =\displaystyle \frac{1}{2}\left(e_1+e_2\right), &  \text{ for }i\in \{3,4\}, \\[.5cm]
  e_i \cdot e_j =0, & \text{ for } i\neq j. 
\end{array}\right.
$$

Let $g:\A(K_{1,1,2})\longrightarrow \A_{RW}(K_{1,1,2})$ be an evolution homomorphism such that for any $i\in\{1,2,3,4\}$ 
$$g(e_i )= \sum _{k=1} ^{4}t_{ik}e_k,$$
where the $t_{ik}$'s are scalars. Then
\[
\begin{array}{lllllll}
g(e_i)\cdot g(e_j)
                   &=&\displaystyle \left(\frac{t_{i2}t_{j2}}{3}+\frac{t_{i3}t_{j3}}{2}+\frac{t_{i4}t_{j4}}{2}\right)e_1+\left(\frac{t_{i1}t_{j1}}{3}+\frac{t_{i3}t_{j3}}{2}+\frac{t_{i4}t_{j4}}{2}\right)e_2 \,\,+ \\[.5cm]
	                 &&\displaystyle\left(\frac{t_{i1}t_{j1}}{3} + \frac{t_{i2}t_{j2}}{3}\right)e_3 + \left(\frac{t_{i1}t_{j1}}{3} +\frac{t_{i2}t_{j2}}{3}\right)e_4  .\\ 
\end{array}
\]
As $g$ is an evolution homomorphism, we have $g(e_i)\cdot g(e_j)=0$ for any $i\not = j$. This implies
\begin{eqnarray}
\label{ex:1}\frac{t_{i2}t_{j2}}{3}+\frac{t_{i3}t_{j3}}{2}+\frac{t_{i4}t_{j4}}{2}   &=& 0,\\
\label{ex:2}\frac{t_{i1}t_{j1}}{3}+\frac{t_{i3}t_{j3}}{2}+\frac{t_{i4}t_{j4}}{2}  &=& 0,\\
\label{ex:3}t_{i1}t_{j1} +t_{i2}t_{j2}&=& 0.
\end{eqnarray}

Adding (\ref{ex:1}) and (\ref{ex:2})  and using  (\ref{ex:3})  we obtain $t_{i3}t_{j3} +t_{i4}t_{j4}=0$ so we can assert, again by (\ref{ex:1}) and (\ref{ex:2}) that
\smallskip
\begin{equation}\label{ex:primera}
t_{i2}t_{j2}=0 \text{ and } t_{i1}t_{j1}=0,\,\,\,  \text{ for } i,j \in \{ 1,2,3,4\} \text{ and } i\not=j.
\end{equation}
\smallskip

Also  we have 
$$
g(e_1)\cdot g(e_1)= \left( \frac{t_{12}^2 }{3} + \frac{t_{13}^2 }{2} +\frac{t_{14}^2 }{2} \right)e_1 + \left( \frac{t_{11}^2 }{3} + \frac{t_{13}^2 }{2} +\frac{t_{14}^2 }{2} \right)e_2 +  \left( \frac{t_{11}^2 }{3} + \frac{t_{12}^2 }{3} \right)e_3 + \left( \frac{t_{11}^2 }{3} + \frac{t_{12}^2 }{3} \right)e_4.  
$$
and 
\[
\begin{array}{lll}
g(e_1^2)&=& g(e_2+e_3+e_4) = \displaystyle \sum_{k=1}^{4}(t_{2k} + t_{3k} +t_{4k} )e_k, \\
\end{array} \]
Then
\begin{eqnarray}
\label{ex:segunda}t_{21}+t_{31}+ t_{41}&=& \frac{t_{12}^2}{3}+ \frac{t_{13}^2}{2}+ \frac{t_{14}^2}{2}, \\[.2cm]
\label{ex:terceira}t_{22}+t_{32}+ t_{42}&=& \frac{t_{11}^2}{3}+ \frac{t_{13}^2}{2}+ \frac{t_{14}^2}{2},\\[.2cm]
\label{ex:cuatro} t_{23}+t_{33}+ t_{43}&=& t_{24}+t_{34}+ t_{44}\,\,=\,\, \frac{t_{11}^2}{3}+ \frac{t_{12}^2}{3}.
\end{eqnarray}

By applying the same reasoning to $g(e_{i}^2)$ for $i\in\{2,3,4\}$,  we obtain

\begin{eqnarray}
\label{ex:Mtres} t_{11}+t_{31}+ t_{41}&=& \frac{t_{22}^2}{3}+ \frac{t_{23}^2}{2}+ \frac{t_{24}^2}{2}, \\[.2cm]
\label{ex:MMtres} t_{12}+t_{32}+ t_{42}&=& \frac{t_{21}^2}{3}+ \frac{t_{23}^2}{2}+ \frac{t_{24}^2}{2}, \\[.2cm]
\label{ex:Mcuatro} t_{11}+t_{21}&=& \frac{t_{32}^2}{3}+ \frac{t_{33}^2}{2}+ \frac{t_{34}^2}{2}= \frac{t_{42}^2}{3}+ \frac{t_{43}^2}{2}+ \frac{t_{44}^2}{2}, \\[.2cm]
\label{ex:Mcinco}  t_{12}+t_{22}&=& \frac{t_{31}^2}{3}+ \frac{t_{33}^2}{2}+ \frac{t_{34}^2}{2}= \frac{t_{41}^2}{3}+ \frac{t_{43}^2}{2}+ \frac{t_{44}^2}{2}, \\[.2cm]
\label{ex:Mseis}  t_{14}+t_{24} \,\,= \,\, t_{13}+t_{23} &=& \frac{t_{31}^2}{3}+ \frac{t_{32}^2}{3}\,\,=\,\, \frac{t_{41}^2}{3}+ \frac{t_{42}^2}{3}, 
\end{eqnarray}

\noindent
where Equations \eqref{ex:Mtres} and \eqref{ex:MMtres} are coming from $g(e_2^2)$, while Equations \eqref{ex:Mcuatro}-\eqref{ex:Mseis} are coming from $g(e_3^2)$ and $g(e_4^2)$. By \eqref{ex:primera}  we have four possible cases: 

\smallskip
\noindent
{\bf \underline{Case 1}: $t_{11}=t_{21}=t_{31}=t_{41}=0$.}  In this case, the left side of Equations  \eqref{ex:segunda} and \eqref{ex:Mtres} are both equal to zero, and $g(e_{1})=g(e_{2})=0$. This together with \eqref{ex:Mcuatro} implies $g=0$, i.e. $g $ is the null map.

\smallskip
\noindent
{\bf \underline{Case 2}: $t_{31}\neq 0$ or $t_{41}\neq 0$.} In this case the left side of Equations \eqref{ex:segunda} and \eqref{ex:Mtres} are not zero. Then $t_{11}=t_{21}=0$ and by (\ref{ex:Mcuatro}) 
$$ t_{32}=t_{33} =  t_{34}=t_{42}=  t_{43}=t_{44}=0,$$
which implies by (\ref{ex:Mcinco}) $ t_{31}^2=t_{41}^2$. Therefore, by (\ref{ex:primera}), $t_{31}=t_{41}=0$ and we get a contradiction. So it should be $t_{31}=t_{41}=0$.

\smallskip
\noindent
{\bf \underline{Case 3}: $t_{11}\neq 0$ and $t_{21}=t_{31}=t_{41}=0$.} Now, the left side of Equation  \eqref{ex:segunda} is zero while the left side of Equation \eqref{ex:Mtres} is not zero. Then by \eqref{ex:segunda}
$$t_{12}=  t_{13}=t_{14}=0,$$
and $t_{32}^2=t_{42}^2$ by (\ref{ex:Mseis}). Therefore, by (\ref{ex:primera}), 
\begin{equation*} \label{ex:paula}
 t_{32}=t_{42}=0,
\end{equation*}
 and by (\ref{ex:MMtres}), we get $t_{23}=t_{24}=0$. As we are assuming $t_{11} \neq 0$ then, by  (\ref{ex:terceira}), $t_{22} \neq 0$. This in (\ref{ex:terceira}) and (\ref{ex:Mtres}) implies  
$$ t_{22}= \frac{t_{11}^2}{3} \, \, \, \, \textrm{and} \, \, \, \,  t_{11}= \frac{t_{22}^2}{3},$$
so $t_{11}=t_{22}=3.$ By (\ref{ex:Mcuatro}) 
$$t_{33}^2+t_{34}^2=t_{43}^2+t_{44}^2=6,$$ and then $$(t_{33}^2+t_{34}^2) +(t_{43}^2+t_{44}^2)=12.$$
On the other hand, by (\ref{ex:cuatro})
$$ t_{33}+t_{43}=t_{34}+t_{44}=3,$$
thus $t^2_{33} + 2t_{33}t_{43} +t^2_{43}= 36$, $t^2_{34} + 2t_{34}t_{44} +t^2_{44}= 9$, and adding terms we get
$$
(t^2_{33}+t^2_{34})+(t^2_{43}+t^2_{44})+ 2(t_{33}t_{43}+t_{34}t_{44} )=18.
$$
We know, for $i=3$ and $j=4$ in \eqref{ex:1} and \eqref{ex:primera}, that $t_{33}t_{43}+t_{34}t_{44}=0$, hence $$(t^2_{33}+t^2_{34})+(t^2_{43}+t^2_{44})=18,$$ 
and we get a contradiction. Therefore, it should be $t_{11}=0$.

\smallskip
\noindent
{\bf \underline{Case 4}: $t_{21}\neq 0$ and $t_{11}=t_{31}=t_{41}=0$.} The computations in this case follows as in the Case 3, but now the left side of Equation \eqref{ex:Mtres} is zero while the left side of Equation \eqref{ex:segunda} is not zero. By proceeding as before we can conclude that it should be $t_{21}=0$. 

\smallskip
Finally, we conclude that the only possible case is the Case 1, and therefore the only evolution homomorphism between $\A(K_{1,1,2})$ and $\A_{RW}(K_{1,1,2})$ is the null map. This in turns implies $\A(K_{1,1,2}) \ncong \A_{RW}(K_{1,1,2})$ as evolution algebras.
\end{exa}

Although we believe that a necessary and sufficient condition for $\mathcal{A}(K_{a_1, a_2, \ldots, a_n})\cong \mathcal{A}_{RW}(K_{a_1, a_2, \ldots, a_n})$ is $n=2$ or $a_i$'s to be equal; a proof for it seems to require more work, and therefore it remains as an interesting open problem for future investigation.

\section*{Acknowledgements}

This work was supported by FAPESP [grant numbers 2015/03868-7, 2016/11648-0]; and CNPq [grant numbers 235081/2014-0, 304676/2016-0]. Part of this work was carried out during a stay of P.C. and P.M.R. at the Universit\'e Paris-Diderot, and a visit at the Universidad de Antioquia.  The authors wishes to thank these institutions for the hospitality and support. Special thanks are given to the referee, whose careful reading of the manuscript and valuable comments contributed to improve this paper.

\bigskip

\end{document}